\documentclass[pdflatex,sn-mathphys-num]{sn-jnl}

\usepackage{graphicx}
\usepackage{multirow}
\usepackage{amsmath,amssymb,amsfonts}
\usepackage{subcaption}
\usepackage{amsthm}
\usepackage{mathrsfs}
\usepackage{float}
\usepackage[title]{appendix}
\usepackage{xcolor}
\usepackage{textcomp}
\usepackage{manyfoot}
\usepackage{booktabs}
\usepackage{algorithm}
\usepackage{algorithmicx}
\usepackage{algpseudocode}
\usepackage{listings}
\usepackage{diagbox}
\usepackage[shortlabels]{enumitem}
\usepackage{booktabs} 
\usepackage{xcolor}
\usepackage{rotating}
\usepackage{caption}

\makeatletter
\let\orcidlogo\@undefined
\makeatother
\usepackage{orcidlink}

\theoremstyle{thmstyleone}
\newtheorem{theorem}{Theorem}[section]
\newtheorem{prop}[theorem]{Proposition}

\newtheorem{lemma}[theorem]{Lemma}

\theoremstyle{remark}
\newtheorem{remark}[theorem]{Remark}
\theoremstyle{thmstylethree}
\newtheorem{definition}[theorem]{Definition}

\newcommand{\espe}{\mathbb{E}}
\newcommand{\prob}{\mathbb{P}}
\newcommand{\var}{\operatorname{Var}}

\newcommand{\dpareto}{\operatorname{Pareto}}

\newcommand{\dgeom}{\operatorname{Geom}^*}
\newcommand{\dfrechet}{\operatorname{Fr\acute{e}chet}}
\newcommand{\dloglogistic}{\operatorname{Log-logistic}}

\newcommand{\e}{\operatorname{e}}

\newcommand{\toas}{\xrightarrow{\text{a.s.}}}

\newcommand{\R}{\operatorname{\mathbb{R}}}
\newcommand{\Z}{\operatorname{\mathbb{Z}}}

\newcommand{\N}{\operatorname{\mathbb{N}}}

\begin{document}

\title[Estimating the tail index from geometric records]{Estimating the tail index of Pareto-type distributions from geometric records}

\author[1,2]{\fnm{Mart\'{i}n} \sur{Alcalde}\,\orcidlink{0009-0007-4275-2998}}\email{malcalde@unizar.es}
\equalcont{These authors contributed equally to this work.}

\author[3]{\fnm{Ra\'ul} \sur{Gouet}\,\orcidlink{0000-0003-3125-6773}}\email{rgouet@dim.uchile.cl}
\equalcont{These authors contributed equally to this work.}

\author*[1,2]{\fnm{Miguel} \sur{Lafuente}\,\orcidlink{0000-0001-8471-3224}}\email{miguellb@unizar.es}
\equalcont{These authors contributed equally to this work.}

\author[1,2]{\fnm{F. Javier} \sur{L\'{o}pez}\,\orcidlink{0000-0002-7615-2559}}\email{javier.lopez@unizar.es}
\equalcont{These authors contributed equally to this work.}

\author[1,2]{\fnm{Gerardo} \sur{Sanz}\,\orcidlink{0000-0002-6474-2252}}\email{gerardo.sanz@unizar.es}
\equalcont{These authors contributed equally to this work.}

\affil[1]{\orgdiv{Departamento de M\'{e}todos Estad\'{i}sticos}, \orgname{Facultad de Ciencias, Universidad de Zaragoza}, \orgaddress{\street{C/ Pedro Cerbuna, 12}, \city{Zaragoza}, \postcode{50009}, \country{Spain}}}

\affil[2]{\orgdiv{Instituto de Biocomputaci\'{o}n y F\'{i}sica de Sistemas Complejos (BIFI)}, \orgname{Universidad de Zaragoza}, \orgaddress{\street{C/ Mariano Esquillor, Edificio I + D}, \city{Zaragoza}, \postcode{50018}, \country{Spain}}}

\affil[3]{\orgdiv{Departamento de Ingenier\'{i}a Matem\'{a}tica y Centro de Modelamiento Matem\'{a}tico}, \orgname{Universidad de Chile}, \orgaddress{\street{Av. Beauchef, 851}, \city{Santiago, Regi\'on Metropolitana}, \postcode{8370458}, \country{Chile}}}

\maketitle

\abstract{
In this paper, we develop a novel inferential approach based on geometric records for estimating the tail index of heavy-tailed distributions. We construct a maximum likelihood estimator for the Pareto model and establish strong consistency and asymptotic normality, providing also an explicit expression for the asymptotic variance. These results are then extended to a broad class of Pareto-type distributions. 
The performance of the estimator is assessed via Monte Carlo simulation and compared with classical estimators from the literature. The proposed method is particularly well suited for settings where data arrive sequentially, as it yields smooth estimation trajectories. It is also especially advantageous in applications such as destructive testing, where measuring each item is costly. In this context, the estimator achieves a comparable level of estimation accuracy to Hill's estimator, but with a considerably lower number of fully measured items. An application to the analysis of the distribution of fluctuations of the Dow Jones Industrial Average (DJI) is also presented.}

\keywords{Tail index, Pareto-type distributions, Records, Near-records, Geometric records, Maximum likelihood estimation} 

\section*{Abstract}In this paper, we develop a novel inferential approach based on geometric records for estimating the tail index of heavy-tailed distributions. We construct a maximum likelihood estimator for the Pareto model and establish strong consistency and asymptotic normality, providing also an explicit expression for the asymptotic variance. These results are then extended to a broad class of Pareto-type distributions. 
The performance of the estimator is assessed via Monte Carlo simulation and compared with classical estimators from the literature. The proposed method is particularly well suited for settings where data arrive sequentially, as it yields smooth estimation trajectories. It is also especially advantageous in applications such as destructive testing, where measuring each item is costly. In this context, the estimator achieves a comparable level of estimation accuracy to Hill's estimator, but with a considerably lower number of fully measured items. An application to the analysis of the distribution of fluctuations of the Dow Jones Industrial Average (DJI) is also presented.
\par\medskip

\noindent\textbf{Keywords:} Tail index, Pareto-type distributions, Records, Near-records, Geometric records, Maximum likelihood estimation

\section{Introduction}\label{introdution}

Heavy-tailed distributions arise in many fields, including economics and finance, environmental sciences, and risk analysis. They model phenomena for which extreme values occur more frequently than under classical light-tailed distributions such as the exponential or normal laws. Their ubiquity and the challenges they pose to classical statistical inference have motivated the development of specialized methods for decades, from the seminal works of Hill \cite{Hill75} and Pickands  \cite{Pickands75} to recent contributions \cite{Azeem25,Empacher25,Vogel24,Schmiedt25,She25} .

The behavior of the extreme values drawn from heavy-tailed distributions can be quantified by the tail index, which is defined as follows. Given a cumulative distribution function (CDF) $F$, the tail index of $F$ is the number $\gamma$, if it exists, such that

\begin{equation}
	\lim_{x\to\infty} \frac{1-F(tx)}{1-F(x)} = t^{-\gamma}\;\text{ for all } t>0.\label{def_regularly_varying}
\end{equation}
A CDF satisfying \eqref{def_regularly_varying} for some $\gamma\in(0,\infty)$ is said to be a Pareto-type distribution.

Estimating the tail index $\gamma$ is key for modeling and understanding the abrupt changes that are inherent to Pareto-type distributions. The reader is referred to a current and extensive review of tail index estimators provided in \cite{Fedotenkov20}. These estimators typically rely on upper extremes, exploiting the information contained in the largest observations to infer the heaviness of the tail. Among them, most approaches are based on order statistics, with Hill's estimator (1975) being the most studied in the literature. Recent examples of applications of Hill's estimator include analysis of prices and markets \cite{Begusic18, Davletov22}, risk analysis \cite{Chaudhry25,Gomes20} and hydrology \cite{Langousis16, Nerantzaki22}.

 A different family of tail-index estimators makes use of record-related observations. For instance, Berred \cite{Berred92} proposed estimators based on ordinary record values, whereas Louzaoui and El Arrouchi \cite{Louzaoui20, Louzaoui23} developed estimators based on $k$-record values.

In this paper, we propose an estimator based on geometric records. The notion of geometric records was introduced in \cite{Eliazar2005} to describe observations that exceed a fixed multiple of the current record level; see also \cite{Gouet12Geom}. More precisely, for a constant $\delta>0$, $X_j$ is called a geometric record if $X_j>\delta R_i$, where $R_i$ denotes the current record at the time $X_j$ is observed. We consider the case $\delta\in(0,1)$, which yields more data than standard records.

Our estimator is the MLE of the tail index based on a discretized version of geometric records when the underlying distribution is Pareto. We establish its consistency and asymptotic normality for Pareto-type distributions under mild assumptions.
	
Some features of our estimator are worth highlighting. First, consistency and asymptotic normality hold for every fixed value of $\delta \in(0,1)$. This contrasts with other tail-index estimators such as Hill's, whose consistency requires a tuning parameter to vary with the sample size. Second, the estimator is well suited to destructive testing schemes, in which obtaining a complete measurement necessarily destroys the experimental unit. In these contexts, sampling procedures based on record-type observations provide an attractive inferential framework (see \cite{Glick78}), as they limit full measurements to those observations that are informative for the estimation process. Third, the estimator is appropriate for situations where data arrive sequentially and the tail-index estimate is continuously updated, as in on-line monitoring settings. Since geometric records are incorporated progressively into the estimation procedure as they occur, the resulting estimation paths evolve smoothly with the sample size. This differs from procedures based on upper order statistics, where newly observed data may modify the set of extremes used for inference and consequently induce more abrupt changes in the estimate.

The paper is organized as follows. The remainder of this section introduces the definitions and notation used throughout the paper. Section \ref{section_pareto} analyzes the Pareto case, deriving the maximum likelihood estimator of the tail index and establishing its asymptotic properties. Section \ref{section_asymptotic_properties_general} considers the more general framework of Pareto-type distributions and shows that the asymptotic properties obtained in the Pareto case continue to hold under certain conditions on the tail behavior of the parent distribution.  
Section \ref{section_simulations} presents simulation results that illustrate the finite-sample performance of the proposed estimator. Our estimator is also advantageous relative to Hill's estimator, particularly in destructive-testing settings, where it achieves comparable accuracy while requiring considerably fewer measured units.
We also illustrate a practical use of the estimator through an application to a real data set. Some conclusions and ideas for future work are presented in Section \ref{conclusion}.

\subsection{Notation and preliminaries}\label{section_notation}
Mathematical notation for limits and real-valued functions used throughout the paper is as follows. Sequences are denoted by $\{a_n\}_{n\in\N}$. Given two sequences $\{a_n\}_{n\in\N}$ and $\{b_n\}_{n\in\N}$, we write $a_n\sim b_n$ to denote asymptotic equivalence, that is, $\lim_{n\to\infty} a_n/b_n = 1$. We also use the standard Landau notation $\mathrm{O}(\cdot)$ and $\mathrm{o}(\cdot)$. For a real number $x$, its floor and ceiling are denoted by $\lfloor x\rfloor$ and $\lceil x \rceil$, respectively, and its fractional part by $\{x\}$; in other words, $\{x\}:= x-\lfloor x\rfloor$. We write the positive part of $x$ as $x_+$. The indicator function is denoted by $\mathbf{1}(\cdot)$.

Regarding probabilistic notation, all random variables are defined on a common complete probability space $(\Omega,\mathcal{F},\prob)$. The expectation operator is denoted by $\espe(\cdot)$. Given a sub-$\sigma$-algebra $\mathcal{H}\subseteq\mathcal{F}$, conditional probability and conditional expectation are written as $\prob(\cdot\mid\mathcal{H})$ and $\espe(\cdot\mid\mathcal{H})$, respectively. Convergence in distribution and almost sure convergence are denoted by $\Rightarrow$ and $\toas$. The notation $\dgeom(p)$ stands for the geometric distribution supported on $\Z_+=\{0,1,2,\ldots\}$ with success probability $p$. In addition, given a Borel set $B$ and a random variable $X$ such that $\prob(X\in B)>0$, we will refer to the conditional distribution of $X$ on the event $\{X\in B\}$ as the truncated distribution of $X$ to $B$.

In what follows, $\{X_n\}_{n\in\N}$ denotes a sequence of independent and identically distributed (i.i.d.) nonnegative random variables defined on the above probability space, with continuous distribution function $F$ and survival function $\overline F:=1-F$. When the survival function satisfies
\begin{equation}
	\overline F(x) = \left(\frac{x}{D}\right)^{-\gamma},\quad x>D,\label{def_pareto}
\end{equation}
for some constant $D>0$, we say that the random variables $X_i$ have a Pareto distribution and we write $X_i \sim \dpareto(\gamma; D)$. 
A function $L:(0,\infty)\to(0,\infty)$ is said to be \emph{slowly varying at infinity} if $
\lim_{x\to\infty} L(tx)/L(x)=1 $
for every $t>0$. Note that if the random variables $X_i$ have a Pareto-type distribution, their survival function can be written as
\begin{equation}
	\overline F(x)=x^{-\gamma}L(x),\quad  x>D,\label{def_paretotype}
\end{equation}
where $L$ is slowly varying at infinity.

The sequence of partial maxima $\{M_n\}_{n\in\Z_+}$ is defined as $M_n:=\max\{X_1,\ldots,X_n\}$ for $n\in\N$, and we set $M_0=-1$ by convention. 
The sequence of record times $\{L_n\}_{n\in\N}$ is defined as $L_1:=1$ and, for $n\ge1$, $L_{n+1}:=\inf\{j>L_n: X_j>X_{L_n}\}$. The continuity of $F$ ensures that the sequence $\{L_n\}_{n\in\N}$ is almost surely well defined. 
The sequence of record values $\{R_n\}_{n\in\N}$ is given by $R_n:=X_{L_n}$ for $n\in\N$.

\begin{definition}
	Let $\delta\in(0,1)$ and $n\in\N$. We say that $X_n$ is a \emph{geometric record} if $X_n>\delta M_{n-1}$ and $X_n$ is a \emph{geometric near-record} if $X_n\in(\delta M_{n-1},M_{n-1}]$.
	\label{def_delta_record_mult}
\end{definition}

As $\delta\in(0,1)$, an observation is a geometric record if and only if it is either a record or a geometric near-record. The observations $X_j\in(\delta R_i,R_i]$, with $L_{i}<j<L_{i+1}$, are the geometric near-records which appear while the current record value is $R_i$. We say that these geometric near-records are \emph{associated} with record $R_i$. We write $\mathbf{R}=(R_1,\ldots,R_n)$ and $\mathbf{S}=(S_1,\ldots,S_n)$, where $S_i$ is the number of geometric near-records associated with $R_i$.

Before introducing further notation, we point out that, since our primary goal is to estimate the tail index $\gamma$ in Pareto-type distributions, collecting information on small observations may be disadvantageous: such values lie outside the tail and contribute little or no information about $\gamma$. Therefore, it is reasonable to consider only those records that exceed a fixed threshold $A>0$, together with their associated geometric near-records.

We define the discretization of the geometric near-records as follows: for each record $R_i$, the interval $(\delta R_i, R_i]$ is split into $m$ disjoint subintervals $(a^v \delta R_i, a^{v+1} \delta R_i]$, for $v=0,1,\ldots,m-1$, where $a=\delta^{-1/m}$ and $m\ge 2$ is a fixed integer. Clearly, each geometric near-record associated with $R_i$ falls in one of these subintervals. If $S_i>0$, we define $V_i^j$ as the index $v$ of the subinterval containing the $j$-th geometric near-record associated with $R_i$, where $j=1,\ldots,S_i$.

Similarly, we discretize the record values: for $i\ge 1$, record $R_{i} \in (R_{i-1}, \infty)$ belongs to the interval $(a^k R_{i-1}, a^{k+1} R_{i-1}]$ for some $k\ge 0$, and we define $K_{i}$ as the corresponding subinterval index; that is, $K_{i}=\left\lceil\log_a\left(\frac{R_i}{R_{i-1}}\right)\right\rceil-1$, with $R_0:= A$ by convention.
Finally, we collect all indices into vectors, namely, $\mathbf{K} = (K_1,\ldots,K_n)$ for records and 
$
\mathbf{V} = (V_1^1,\ldots,V_1^{S_1}, \ldots, V_n^1, \ldots, V_n^{S_n})$ 
for geometric near-records. Thus, our sample for the rest of the paper will be 
\begin{equation}
	\mathbf{T}_n=(\mathbf{K}_n,\mathbf{S}_n,\mathbf{V}_n),\label{eq_muestra}
\end{equation} 
collecting the information about the geometric records from the first $n$ records, once an observation greater than $A$ has appeared. When there is no risk of confusion we omit the subscript $n$ from the notation above.

\section{Estimator of the tail index for the Pareto distribution} \label{section_pareto}

In this section, we construct an estimator of the tail index for observations following the Pareto distribution given in \eqref{def_pareto}, where $D>0$ is a known constant, and prove its consistency and asymptotic normality.

In the next proposition, the distribution of the vectors of the sample $\mathbf{T}$ is established.

\begin{prop}\label{prop_distribution_t}
	Assume that $X_1\sim\dpareto(\gamma; D)$. Let $\delta\in(0,1)$, and also let $n\ge 1,m\ge 2$ be integers and $A\ge \delta^{-1}D$. Then:
	\begin{enumerate}[a)]
		\item The random variables $\mathbf{S}=(S_1,\ldots,S_n)$ are i.i.d.~with distribution $\dgeom(\delta^\gamma)$ and are independent of the record values $\mathbf{R}$. \label{prop_distribution_t_si}
		\item \label{prop_distribution_t_b} Let $i=1,\ldots,n$ and $s_i\ge1$ be an integer. Conditioned on  $\{S_i = s_i\}$, the random variables $V_i^1,\ldots,V_i^{s_i}$ are i.i.d.~following a truncated geometric distribution to  $\{0,1,\ldots,m-1\}$ with success probability $1 - \delta^{\frac{\gamma}{m}}$. Furthermore, they are  independent of the record values $\mathbf{R}$. \label{prop_distribution_t_vij}
		\item The random variables $\mathbf{K}=(K_1,\ldots,K_{n})$ are i.i.d.~with distribution $\dgeom(1-\delta^{\frac{\gamma}{m}})$ and are independent of the vectors $\mathbf{S}$ and $\mathbf{V}$. \label{part_K}
	\end{enumerate}
	\begin{proof}
		Reasoning as in Propositions 2.1-2.3 in \cite{Gouet14}, we have that, conditioned on $\{R_i=r_i\}$, with $r_i>A$, the number of geometric near-records associated with $R_i$ is independent of the number of geometric near-records associated with other records and of their values. Moreover, for $r_i>A$, $s_i\ge1$, 
		\begin{equation*}
			\prob(S_i = s_i\mid R_i = r_i)=\left(\frac{\overline{F}(\delta r_i) - \overline{F}(r_i)}{\overline{F}(\delta r_i)}\right)^{s_i}\frac{\overline{F}(r_i)}{\overline{F}(\delta r_i)}=(1-\delta^\gamma)^{s_i}\delta^\gamma,
		\end{equation*}
		thus proving $a)$.
		
		Analogously, given $R_i=r_i$, with $r_i>A$, the values of the geometric near-records associated with record $R_i$ are independent of their number $S_i$ and the values of other records and geometric near-records, with 
		\begin{align*}
			\prob(V_i^j = v_i^j \mid S_i = s_i, R_i = r_i) &= \frac{\overline{F}(a^{v_i^j}\delta r_i) - \overline{F}(a^{v_i^j + 1}\delta r_i)}{\overline{F}(\delta r_i)-\overline{F}(r_i)} = \frac{\delta^{v_i^j\frac{\gamma}{m}}(1 - \delta^{\frac{\gamma}{m}})}{1-\delta^{\gamma}},
		\end{align*}
		for all $v_i^j\in\{0,1,\ldots,m-1\}$. This proves $b)$. 
		
		For $c)$, note that the random variables $K_1,\ldots,K_n$ are functions of $(R_1,\ldots,R_n)$ and are therefore independent of $\mathbf{S}$ and $\mathbf{V}$ by $a)$ and $b)$.  Also
		\begin{equation}
			K_i = \bigg\lceil\log_a{\bigg(\frac{R_{i}}{R_{i-1}}\bigg)}\bigg\rceil - 1= \bigg\lfloor\log_a{\bigg(\frac{R_{i}}{R_{i-1}}\bigg)}\bigg\rfloor\quad\text{a.s.},\label{eq_K}\quad i=1,\ldots,n.
		\end{equation}
		In addition, from  \eqref{eq_K}, the variables $K_1,\ldots,K_{n}$ are i.i.d.~since the sequence of quotients $\{R_{i}R_{i-1}^{-1}\}_{i\in\N}$ is i.i.d. following a $\dpareto(\gamma; 1)$ distribution; see \cite{Arnold98}, p. 22. Moreover, for each $k_i\in\Z_+$,
		\begin{equation*}
			\prob(K_i = k_i) = \prob\left(\frac{R_{i}}{R_{i-1}}\in(a^{k_i},a^{k_i+1}]\right)= \delta^{k_i\frac{\gamma}{m}}(1 - \delta^{\frac{\gamma}{m}}),
		\end{equation*}
		proving $c)$.
	\end{proof}
\end{prop}

\begin{remark}By Proposition \ref{prop_distribution_t}, the components of the sample $\mathbf{T}$ associated with each record $R_i$, that is, $(K_i;S_i;V_i^1,\ldots,V_i^{S_i})$, for $i=1,\ldots,n$, are independent and identically distributed. \label{remark_independence}
\end{remark}

\begin{theorem} \label{thm_MLE}
	Assume that $X_1\sim\dpareto(\gamma; D)$. Let $\delta\in(0,1)$, and also let $n\ge 1,m\ge 2$ be integers and $A\ge \delta^{-1}D$. Then, the MLE of $\gamma$ based on the sample $\mathbf{T}$ in \eqref{eq_muestra} is given by
	\begin{equation}
		\hat{\gamma}_{\delta,m,n}^A := 
		m\log_\delta{(\hat{\beta}_{\delta,m,n}^A)},
		\label{eq_estimator_gamma}
	\end{equation}
	with
	\begin{equation}
		\hat{\beta}_{\delta,m,n}^A := 
		\dfrac{mn + \sum\limits_{i=1}^n\sum\limits_{j=1}^{S_i}V_i^j + \sum\limits_{i=1}^{n}K_i}{(m+1)n + \sum\limits_{i=1}^n\sum\limits_{j=1}^{S_i}V_i^j + \sum\limits_{i=1}^{n}K_i+\sum\limits_{i=1}^n S_i}.\label{eq_estimator_beta}
	\end{equation}
\end{theorem}
\begin{proof}
	Let $\mathbf{t} = (\mathbf{k},\mathbf{s}, \mathbf{v})$ be a realization of the sample $\mathbf{T}$. By Proposition \ref{prop_distribution_t}, the likelihood at $\mathbf{t}$ is 
	\begin{equation}
		\mathcal{L}(\mathbf{t;\gamma}) = \prob(\mathbf{K}=\mathbf{k})\ \prob(\mathbf{S}=\mathbf{s}, \mathbf{V}=\mathbf{v})
		=
		(\delta^{\gamma})^{n + \frac{1}{m}\sum\limits_{i=1}^n\sum\limits_{j=1}^{s_i}v_i^j+\frac{1}{m}\sum\limits_{i=1}^{n}k_i}(1-\delta^{\frac{\gamma}{m}})^{n +\sum\limits_{i=1}^n s_i}.
		\label{eq_2}
	\end{equation}
	Using the reparametrization $\beta = \delta^{\frac{\gamma}{m}}$, the expression \eqref{eq_2} becomes
	$$
	\mathcal{L}(\mathbf{t};\beta) =
	\;\beta^{mn + \sum\limits_{i=1}^n\sum\limits_{j=1}^{s_i}v_i^j+\sum\limits_{i=1}^{n}k_i}(1-\beta)^{n  +\sum\limits_{i=1}^n s_i},
	$$
	so the MLE of $\beta$ is given in \eqref{eq_estimator_beta}. The MLE of $\gamma$  in \eqref{eq_estimator_gamma} is obtained by plugging the MLE $\hat{\beta}^A_{\delta,m,n}$ into the  relationship between $\beta$ and $\gamma$.
\end{proof}

\begin{remark}
The assumption $A\ge \delta^{-1}D$ in Proposition~\ref{prop_distribution_t} and Theorem~\ref{thm_MLE} is needed to prove the stated results. If this condition is not satisfied, the distributions of $S$ and $V$ are more complex, and the maximum likelihood estimator has no closed-form.

\end{remark}

\begin{remark}\label{D_descon} In Theorem \ref{thm_MLE} we assume that $D$ is known. However, the MLE of $\gamma$ in \eqref{eq_estimator_gamma} does not explicitly depend on $D$. Therefore, the same expression for the MLE remains valid even when $D$ is unknown, provided that the threshold $A$ satisfies $A \ge \delta^{-1} D$.
\end{remark}

\subsection{Asymptotic properties of the estimator for Pareto samples} \label{section_asymptotics_pareto}
In this section, we study the asymptotic properties of $\hat{\gamma}_{\delta,m,n}^A$ as the number of records $n$ tends to infinity.

We first establish its strong consistency. To this end, we need some results on the moments of the statistics of the sample $\mathbf{T}$.

\begin{lemma} \label{lemma_sum_v}
	Assume that $X_1\sim\dpareto(\gamma; D)$. Let $\delta\in(0,1)$, and also let $n\ge 1,m\ge 2$ be integers and $A\ge \delta^{-1}D$. Then:
	\begin{enumerate}[a)] 
		\item For $s\ge1$:
		\begin{equation*}
			\espe(V_1^1\mid S_1 = s) = \frac{\delta^{\frac{\gamma}{m}}}{1-\delta^{\frac{\gamma}{m}}} -m\frac{\delta^\gamma}{1-\delta^\gamma},\quad
			\var(V_1^1\mid S_1 = s) = \frac{\delta^{\frac{\gamma}{m}}}{(1-\delta^{\frac{\gamma}{m}})^2} - m^2\frac{\delta^\gamma}{(1-\delta^\gamma)^2}.
		\end{equation*}
		\item The components of the vector
		$\left(\sum_{j=1}^{S_1}V_1^j,\ldots, \sum_{j=1}^{S_n}V_n^j\right)$ are i.i.d. with expectation
		\begin{equation}\espe\bigg(\sum_{j=1}^{S_1}V_1^j\bigg) = \frac{1-\delta^\gamma}{(1-\delta^{\frac{\gamma}{m}})\delta^{\gamma(1-\frac{1}{m})}}-m.
			\label{lemma_sum_v_expectation}\end{equation}
	\end{enumerate}
	\begin{proof}
		Part $a)$ is an immediate consequence of Proposition \ref{prop_distribution_t} and Lemma \ref{lemma_truncated_geometric} in the Appendix. For part $b)$, note that the components of the random vector are i.i.d.~from Propositions
		\ref{prop_distribution_t}, $a)$ and $b)$. Now, since $\espe(V_1^1\mid S_1 = s)$ does not depend on $s$ provided that $s\ge 1$, $\espe(\sum_{j=1}^{S_1}V_1^j) =\espe(\sum_{j=1}^{S_1}V_1^j\mathbf{1}(S_1 \ge 1)) = \espe(S_1)\espe(V_1^1\mid S_1 \ge 1)$, which yields \eqref{lemma_sum_v_expectation}.
	\end{proof}
\end{lemma}

\begin{theorem}
	\label{theorem_LLN}Assume that $X_1\sim\dpareto(\gamma; D)$. Let $\delta\in(0,1)$, $m\ge 2$ be an integer, and $A\ge \delta^{-1}D$. Then,
	\begin{equation}\label{forLLN}
		\hat{\gamma}_{\delta, m,n}^A\toas \gamma\;\text{ as }\; n\to\infty.
	\end{equation}
	\begin{proof}
		Note that \eqref{forLLN} is equivalent to $\hat{\beta}^A_{\delta,m,n}\toas\delta^{\frac{\gamma}{m}}$ as $n\to\infty$. From \eqref{eq_estimator_beta}, we have
		$$
		\hat{\beta}_{\delta,m,n}^A = \frac{m + \frac{1}{n}\sum\limits_{i=1}^n\sum\limits_{j=1}^{S_i}V_i^j + \frac{1}{n}\sum\limits_{i=1}^{n}K_i}{m+1 + \frac{1}{n}\sum\limits_{i=1}^n\sum\limits_{j=1}^{S_i}V_i^j + \frac{1}{n}\sum\limits_{i=1}^{n}K_i+\frac{1}{n}\sum\limits_{i=1}^n S_i}.
		$$
		Given that, by Proposition \ref{prop_distribution_t} and Lemma \ref{lemma_sum_v}, each of the random vectors $(\sum_{1\le j\le S_1} V_1^j,\ldots,\sum_{1\le j\le S_n} V_n^j)$, $(S_1,\ldots, S_n)$, and $(K_1,\ldots, K_n)$ are i.i.d.,  we have
		\begin{equation*}
			\hat{\beta}_{\delta,m,n}^A\toas \frac{m + \espe\bigg(\sum\limits_{j=1}^{S_1}V_1^j\bigg) + \espe(K_1)}{m + 1+ \espe\bigg(\sum\limits_{j=1}^{S_1}V_1^j\bigg) + \espe(K_1) + \espe(S_1)}\;\text{ as }n\to\infty.
		\end{equation*}
		By replacing $\espe(S_1)$,  $\espe(K_1)$ and $\espe(\sum_{1\le j\le S_1}V_1^j)$ with their values (see  
		Proposition \ref{prop_distribution_t} parts $a)$ and $c)$ and Lemma \ref{lemma_sum_v}), the result follows.
	\end{proof}
\end{theorem}

We now prove the asymptotic normality of our estimator, also finding its asymptotic variance. We begin with a technical lemma.

\begin{lemma}\label{lemma_a}
	Assume that $X_1\sim\dpareto(\gamma; D)$. Let $\delta\in(0,1)$, and also let $n\ge 1,m\ge 2$ be integers and $A\ge \delta^{-1}D$. Define 
	\begin{equation}\label{defU}
		U_i := \sum_{j=1}^{S_i} V_i^j - \frac{\delta^{\frac{\gamma}{m}}}{1-\delta^{\frac{\gamma}{m}}}S_i + K_i,\quad i=1,\ldots,n.
	\end{equation} 
	Then $U_1,\ldots,U_n$ are i.i.d. random variables with 
	\begin{equation*}
		\espe(U_1) = \frac{\delta^{\frac{\gamma}{m}}}{1-\delta^{\frac{\gamma}{m}}}-m,\qquad
		\var(U_1) = \frac{1}{(1-\delta^{\frac{\gamma}{m}})^2\delta^{\gamma(1-\frac{1}{m})}}.
	\end{equation*} 
	\begin{proof}Both the i.i.d.~property and the expression for the expectation follow directly from Proposition \ref{prop_distribution_t}, Remark \ref{remark_independence} and Lemma \ref{lemma_sum_v}. As for the variance, we have
		{\allowdisplaybreaks
			\begin{align*}
				\var(U_1) &= \var\bigg(\bigg(\sum_{j=1}^{S_1} V_1^j - \frac{\delta^{\frac{\gamma}{m}}}{1-\delta^{\frac{\gamma}{m}}}S_1\bigg)\mathbf{1}(S_1\ge 1)\bigg) + \var(K_1)\\
				&= \var\bigg(\espe\bigg(\sum_{j=1}^{S_1} V_1^j - \frac{\delta^{\frac{\gamma}{m}}}{1-\delta^{\frac{\gamma}{m}}}S_1\ \bigg\vert\  S_1\bigg)\mathbf{1}(S_1\ge 1)\bigg)\\
				&\quad + \espe\bigg(\var\bigg(\sum_{j=1}^{S_1} V_1^j - \frac{\delta^{\frac{\gamma}{m}}}{1-\delta^{\frac{\gamma}{m}}}S_1\ \bigg\vert\ S_1\bigg)\mathbf{1}(S_1\ge 1)\bigg)
				 + \frac{\delta^{\frac{\gamma}{m}}}{(1-\delta^{\frac{\gamma}{m}})^2}\\
				&= \var\bigg(S_1\bigg(\espe(V_1^1\mid S_1) - \frac{\delta^{\frac{\gamma}{m}}}{1-\delta^{\frac{\gamma}{m}}}\bigg)\mathbf{1}(S_1\ge 1)\bigg)\\
				&\quad+ \espe(S_1\var(V_1^1\mid S_1)\mathbf{1}(S_1\ge 1)) + \frac{\delta^{\frac{\gamma}{m}}}{(1-\delta^{\frac{\gamma}{m}})^2}\\
				&=m^2\frac{\delta^{2\gamma}}{(1-\delta^\gamma)^2}\var(S_1)+ \bigg(\frac{\delta^{\frac{\gamma}{m}}}{(1-\delta^{\frac{\gamma}{m}})^2} - m^2\frac{\delta^\gamma}{(1-\delta^\gamma)^2}\bigg)\ \espe(S_1) +  \frac{\delta^{\frac{\gamma}{m}}}{(1-\delta^{\frac{\gamma}{m}})^2}\\
				&= \frac{1}{(1-\delta^{\frac{\gamma}{m}})^2\delta^{\gamma(1-\frac{1}{m})}}.
			\end{align*}
		}
	\end{proof}
\end{lemma}

\begin{theorem}\label{theorem_TCL}
	Assume that $X_1\sim\dpareto(\gamma; D)$. Let $\delta\in(0,1)$, $m\ge 2$ be an integer, and $A\ge \delta^{-1}D$. Then,
	\begin{equation*}
		\sqrt{n}(\hat{\gamma}_{\delta, m,n}^A-\gamma)\Rightarrow \sigma_{\delta, m}Z\;\text{ as }\;n\to\infty,\end{equation*}
	where $Z$ is a standard normal random variable and
	\begin{equation}
		\sigma_{\delta,m} := \frac{m(1-\delta^{\frac{\gamma}{m}})\delta^{\frac{\gamma}{2}(1 - \frac{1}{m})}}{-\log{(\delta)}}.\label{asymptotic_variance}
	\end{equation}
	\begin{proof}
		First, we prove the asymptotic normality of   $\hat{\beta}_{\delta,m,n}^A$. Observe that
		\begin{equation}
			\sqrt{n}(\hat{\beta}_{\delta,m,n}^A-\delta^{\frac{\gamma}{m}})=
			(1-\delta^{\frac{\gamma}{m}})\sqrt{n}
			\left(\frac{m + \frac{1}{n}\sum\limits_{i=1}^n\sum\limits_{j=1}^{S_i}V_i^j + \frac{1}{n}\sum\limits_{i=1}^{n}K_i-\frac{\delta^{\frac{\gamma}{m}}}{1-\delta^{\frac{\gamma}{m}}}\frac{1}{n}\sum\limits_{i=1}^n S_i-\frac{\delta^{\frac{\gamma}{m}}}{1-\delta^{\frac{\gamma}{m}}}}{m+1+ \frac{1}{n}\sum\limits_{i=1}^n\sum\limits_{j=1}^{S_i}V_i^j + \frac{1}{n}\sum\limits_{i=1}^{n}K_i+\frac{1}{n}\sum\limits_{i=1}^n S_i}\right).\label{eq_5}
		\end{equation}
		As shown in the proof of Theorem \ref{theorem_LLN}, the denominator on the right-hand side of \eqref{eq_5} converges to $\delta^{-\gamma}(1-\delta^{\frac{\gamma}{m}})^{-1}$ almost surely as $n\to\infty$. Thus, we have
		\begin{align*}
			\sqrt{n}&(\hat{\beta}_{\delta,m,n}^A-\delta^{\frac{\gamma}{m}})\nonumber\\
			&\sim (1-\delta^{\frac{\gamma}{m}})^2\delta^\gamma\sqrt{n}\left(m + \frac{1}{n}\sum\limits_{i=1}^n\sum\limits_{j=1}^{S_i}V_i^j + \frac{1}{n}\sum\limits_{i=1}^{n}K_i-\frac{\delta^{\frac{\gamma}{m}}}{1-\delta^{\frac{\gamma}{m}}}\frac{1}{n}\sum\limits_{i=1}^n S_i-\frac{\delta^{\frac{\gamma}{m}}}{1-\delta^{\frac{\gamma}{m}}}\right)\nonumber\\
			&= (1-\delta^{\frac{\gamma}{m}})^2\delta^\gamma\sqrt{n}\left(\frac{1}{n}\sum_{i=1}^{n} U_i - \frac{\delta^{\frac{\gamma}{m}}}{1-\delta^{\frac{\gamma}{m}}} + m\right)\; \text{ a.s. as }  n\to\infty,
		\end{align*}
		where the variables $U_i$ are defined in \eqref{defU}. By Lemma \ref{lemma_a}, we readily obtain 
		\begin{equation}
			\sqrt{n}(\hat{\beta}_{\delta,m,n}^A-\delta^{\frac{\gamma}{m}})\Rightarrow (1-\delta^{\frac{\gamma}{m}})\delta^{\frac{\gamma}{2}(1+\frac{1}{m})}Z.\label{eq_tcl_beta}
		\end{equation}
		Finally, we apply the Delta Method to \eqref{eq_tcl_beta} taking $g(x):=m\log_\delta(x)$, $x\in(0,1)$, (see, e.g., \cite{Ferguson96}, Theorem 7, p. 45) and the proof is complete.
	\end{proof}
\end{theorem}

\begin{remark}  The asymptotic standard deviation $\sigma_{\delta,m}$ given in \eqref{asymptotic_variance} satisfies
	\begin{equation}
		\sigma_{\delta,m}\in(\gamma\delta^{\frac{\gamma}{2}}, \gamma)
		\label{inequations_asymptotic_sigma}
	\end{equation}
	for all $\delta\in(0,1)$ and $m\ge 2$. A direct calculus check shows that $\sigma_{\delta,m}$ is decreasing in $m$ and increasing in $\delta$. The lower and upper bounds in \eqref{inequations_asymptotic_sigma} correspond to the limits of $\sigma_{\delta,m}$ as $m\to\infty$ and as $\delta\to 1^{-}$, respectively. Regarding the parameter $m$, the larger its value, the finer the discretization of the intervals $(\delta R_i, R_i]$ and $(R_i,\infty)$ becomes. Hence, larger values of $m$ yield less variable estimates. Conversely, the standard deviation $\sigma_{\delta,m}$ must increase as $\delta$ increases, since larger values of $\delta$ imply that fewer geometric near-records are observed, thus increasing the uncertainty. 
\end{remark}

Before concluding this section, we present a result, which shows that, asymptotically, our estimator is not affected by the first statistics of the sample $\mathbf{T}$. This conclusion is natural, since the change of the values of a random but finite number of terms in \eqref{eq_estimator_beta} becomes negligible for large enough $n$. This property will be useful when we address the generalization of Theorems \ref{theorem_LLN} and  \ref{theorem_TCL}.

\begin{prop}\label{prop_AB}
	Assume that $X_1\sim\dpareto(\gamma;D)$, and let $\delta\in(0,1)$, $m\ge2$ be an integer, and $A\ge\delta^{-1}D$.
	Let $\mathbf{T}\equiv\mathbf{T}_{\delta,m,n}^A$ be the associated sample and let $M$ be a random variable taking values in $\N$ almost surely. Consider any modified sample $\mathbf{T}(M)\equiv\mathbf{T}_{\delta,m,n}^A(M)$ that coincides with $\mathbf{T}$
	from index $M$ onward, and let $\hat{\gamma}_{\delta,m,n}^A$ and $\hat{\gamma}_{\delta,m,n}^A(M)$ denote the estimators computed from
	$\mathbf{T}$ and $\mathbf{T}(M)$, respectively. Then,
	\begin{equation*}
		\left\vert\hat{\gamma}_{\delta,m,n}^A(M) - \hat{\gamma}_{\delta,m,n}^A\right\vert
		= \mathrm{O}(n^{-1}) \; \text{ a.s. as } n\to\infty.
	\end{equation*}
	In particular, the limiting behavior of $\hat{\gamma}_{\delta,m,n}^A$ stated in Theorems \ref{theorem_LLN} and  \ref{theorem_TCL} is shared by $\hat{\gamma}_{\delta,m,n}^A(M)$.
\end{prop}
\begin{proof} Let us see first that 
	\begin{equation}\label{difbeta}
		\left\vert\hat{\beta}_{\delta,m,n}^A(M)-\hat{\beta}_{\delta,m,n}^A\right\vert = \mathrm{O}(n^{-1})\; \text{ a.s. as } n\to\infty.
	\end{equation}
	Define
	\begin{equation*}
		\zeta_{n} \equiv \zeta_{\delta,m,n}^A := m + \frac{1}{n}\sum\limits_{i=1}^n\sum\limits_{j=1}^{S_i}V_i^j + \frac{1}{n}\sum\limits_{i=1}^{n}K_i,\qquad
		\xi_{n} \equiv \xi_{\delta,m,n}^A := 1 + \frac{1}{n}\sum_{i=1}^n S_i,
	\end{equation*}
	and let $\zeta_{n}(M) \equiv \zeta_{\delta,m,n}^A(M)$ and $\xi_{n}(M) \equiv \xi_{\delta,m,n}^A(M)$ be their corresponding versions associated with the modified sample $\mathbf{T}(M)$. Then,
	$\hat{\beta}_{\delta,m,n}^A = \zeta_n /(\zeta_n+\xi_n)$ and $\hat{\beta}_{\delta,m,n}^A(M) = \zeta_n(M)/(\zeta_n(M)+\xi_n(M))$. We have
	\begin{align*}
		n\left|\hat{\beta}_{\delta,m,n}^A(M)-\hat{\beta}_{\delta,m,n}^A\right| &= n\frac{| (\zeta_n(M)-\zeta_n)\xi_n(M)-(\xi_n(M)- \xi_n)\zeta_n(M)|}{(\zeta_n(M)+\xi_n(M))(\zeta_n + \xi_n)}\\
		&\le \frac{n\max\{|\zeta_n(M)-\zeta_n|,|\xi_n(M)- \xi_n|\}}{\zeta_n + \xi_n}\\
		&\sim (1-\delta^{\frac{\gamma}{m}})\delta^\gamma n\max\{|\xi_n(M)- \xi_n|, |\zeta_n(M)-\zeta_n|\}\;\text{ a.s. as } n\to\infty.
	\end{align*}
	By the definitions of $\xi_n(M), \xi_n, \zeta_n(M),\zeta_n$, the differences between $\xi_n(M)$ and $\xi_n$, and between $\zeta_n(M)$ and $\zeta_n$ consist of an almost surely finite sum divided by $n$, which implies $\limsup_{n\to\infty} n\max\{|\xi_n(M)- \xi_n|, |\zeta_n(M)-\zeta_n|\}<\infty$ a.s. and \eqref{difbeta} follows. 
	To conclude, it suffices to note that
	$$
	\hat{\gamma}_{\delta,m,n}^A(M) - \hat{\gamma}_{\delta,m,n}^A = m\log_\delta{\left(1+\frac{\hat{\beta}_{\delta,m,n}^A(M) - \hat{\beta}_{\delta,m,n}^A}{\hat{\beta}_{\delta,m,n}^A}\right)} \sim \frac{m}{\delta^{\frac{\gamma}{m}}\log{(\delta)}}(\hat{\beta}_{\delta,m,n}^A(M) - \hat{\beta}_{\delta,m,n}^A)
	$$
	almost surely as $n\to\infty$.
\end{proof}

\begin{remark}\label{remark_threshold_A}
	By Proposition \ref{prop_AB}, modifying a finite number of initial terms of the sample does not affect the consistency and asymptotic normality of the estimator. Consequently, Theorems \ref{theorem_LLN} and \ref{theorem_TCL} remain valid regardless of whether $A \ge \delta^{-1}D$ holds, in spite of the fact that the estimator $\hat{\gamma}_{\delta,m,n}^A$ is no longer the maximum likelihood estimator of $\gamma$ when $A < \delta^{-1}D$. In particular, the asymptotic properties of $\hat{\gamma}_{\delta,m,n}^A$ hold without requiring knowledge of $D$. 
\end{remark}

\section{Asymptotic properties of the estimator for Pareto-type samples}\label{section_asymptotic_properties_general}
In Section \ref{section_asymptotics_pareto}, we described the asymptotic behavior of $\hat{\gamma}^A_{\delta,m,n}$. Now we focus on the extension of those properties to the general setting of Pareto-type distributions, as defined in \eqref{def_paretotype}.

In order to do so, we will transform our sequence $\{X_n\}_{n\in\N}$ of Pareto-type distributed random variables into another sequence $\{X^*_n\}_{n\in\N}$ and use the results in Section \ref{section_pareto} for obtaining the asymptotic behavior of the estimator arising from the latter sequence. Then, we will translate this behavior to the estimator $\hat{\gamma}^A_{\delta,m,n}$ from the original sequence. The main idea of the proof is to show that, under general conditions, these two estimators are not substantially different. We begin with the definition of the sequence $\{X^*_n\}_{n\in\N}$.

Let $\varphi(x) := \overline{F}(x)^{-\frac{1}{\gamma}} = xL(x)^{-\frac{1}{\gamma}}$ for all $x > D$. Given an i.i.d.~sequence $\{X_n\}_{n\in\N}$ with survival function given in \eqref{def_paretotype}, we define a new sequence $\{X^*_n\}_{n\in\N}$ as:
\begin{equation*}
	X^*_n := \varphi(X_n), \quad n \in \N.
\end{equation*}

It is clear that the transformation $\varphi$ is nondecreasing and continuous, and that $\{X^*_n\}_{n\in\N}$ is a sequence of i.i.d.~$\dpareto(\gamma; 1)$  random variables. The record times $\{L^*_i\}_{i\in\N}$ of the sequence $\{X^*_n\}_{n\in\N}$ coincide almost surely with those of $\{X_n\}_{n\in\N}$, that is, $L_i^* = L_i$ a.s. for all $i \in \N$, and the record values $\{R^*_i\}_{i\in\N}$ of the sequence $\{X^*_n\}_{n\in\N}$ are, almost surely,
$$
R^*_i=X^*_{L_i}=\varphi(R_i)=R_iL(R_i)^{-\frac{1}{\gamma}},\quad i\in\N.
$$
The qualification ``almost surely'' is required because $\overline{F}$ is not assumed to be strictly decreasing. However, since ties occur with probability zero, the record structure is preserved with probability $1$.

Regarding geometric records, note that when $R_i$ and $R^*_i$ are the current record values, the observations $X_j$ and $X^*_j$, $j\in\{L_i+1,\ldots,L_{i+1}-1\}$, are geometric records if and only if $X_j > \delta R_i$ and $X^*_j > \delta R^*_i$, respectively, which is equivalent almost surely to $\varphi(X_j) > \varphi(\delta R_i)$ and $\varphi(X_j) > \delta \varphi(R_i)$. 

Let $\mathbf{T}^*=(\mathbf{K}^*,\mathbf{S}^*,\mathbf{V}^*)$ be the sample of geometric records arising from the sequence $\{X^*_n\}_{n\in\N}$ and let $\hat{\gamma}^*_{\delta,m,n}$ be the estimator in Theorem \ref{thm_MLE} based on $\mathbf{T}^*$. The first record in the sample $\mathbf{T^*}$ is $R_1^*=\varphi(R_1)$, meaning that the threshold $A$ in Theorems \ref{theorem_LLN} and \ref{theorem_TCL} is now $A^*=\varphi(A)$. Note that $A^*$ should be greater than $\delta^{-1}$ for a direct application of these theorems, but Remark \ref{remark_threshold_A} asserts that the asymptotic behavior of the estimator in Theorem \ref{thm_MLE} is unchanged when the threshold $A$ is below $\delta^{-1}$. Therefore, we have
\begin{equation}
	\label{convgamma*}	\hat{\gamma}_{\delta, m,n}^*\toas \gamma\;\text{ and }\;\sqrt{n}(\hat{\gamma}_{\delta, m,n}^*-\gamma)\Rightarrow \sigma_{\delta, m}Z\;\text{ as }\;n\to\infty,
\end{equation}
where $Z$ is the standard normal distribution and $\sigma_{\delta, m}$ is given in \eqref{asymptotic_variance}.

Throughout this section, we assume that $L$ is differentiable and there exists some $\eta>0$ such that
\begin{equation}
	\frac{xL'(x)}{L(x)} = \mathrm{o}\ \bigg(\frac{1}{\log{(x)}(\log{\log{(x)}})^{2+\eta}}\bigg)\;\text{ as }\; x\to\infty.\label{eq_condition_L_3}
\end{equation}

\begin{remark} Condition \eqref{eq_condition_L_3} holds for the best-known examples of Pareto-type distributions, namely:
	\begin{itemize}
		\item \emph{the Pareto distribution}, since $L'(x) = 0$ for all $x>D$;
		\item (absolute value) \emph{Student's $t$ distribution}, which satisfies $xL'(x)/L(x)\sim\frac{ \gamma^2(\gamma + 1)}{\gamma + 2}x^{-2}$, with $\gamma$ being the degrees of freedom; 
		\item \emph{the Fr\'{e}chet distribution}, given by $\overline{F}(x) = 1 -\exp{(-x^{-\gamma})}$, $x> 0$,  which satisfies $xL'(x)/L(x)\sim \frac{\gamma}{2}x^{-\gamma}$; 
		\item \emph{the Burr Type III distribution} (\emph{the Dagum distribution}), which has $\overline{F}(x) = 1 - (1+x^{-\gamma})^{-p}$, $x>0$ ($p>0$), and satisfies $xL'(x)/L(x)\sim \gamma\frac{p+1}{2}x^{-\gamma}$;
		\item \emph{the Burr Type XII distribution}, that is, $\overline{F}(x) = (1 + x^\frac{\gamma}{c})^{-c}$, $x> 0$ ($c>0$), which satisfies $xL'(x)/L(x)\sim \gamma x^{-\frac{\gamma}{c}}$; for instance, the log-logistic distribution ($c=1$);
		\item \emph{stable laws} $S(\gamma,\beta,1)$ with $\gamma\in(0,1)$ and $\beta\in(-1,1]$, since $\overline{F}(x;\gamma,\beta,1) = \gamma^{-1}Cx^{-\gamma}+\mathrm{O}(x^{-2\gamma})$ and its probability density function satisfies $f(x;\gamma,\beta,1) = Cx^{-\gamma-1} + \mathrm{O}(x^{-2\gamma-1})$ (see \cite{Nolan20} , Theorem 3.5, p. 75), which yields
		$$
		\frac{xL'(x)}{L(x)} = \frac{\gamma\overline{F}(x;\gamma,\beta,1)-xf(x;\gamma,\beta,1)}{\overline{F}(x;\gamma,\beta,1)} = \mathrm{O}(x^{-\gamma}).
		$$
	\end{itemize}
\end{remark}

Since we aim to prove that the estimators $\hat{\gamma}_{\delta, m,n}^*$ and $\hat{\gamma}_{\delta, m,n}^A$ are not very different for large $n$, we  establish a series of lemmas analyzing the difference between the elements of the sample $\mathbf{T}$ and the elements of $\mathbf{T^*}$. Their proofs, some of them quite technical, are deferred to Section \ref{proofLemmas}.

Let $\mathcal{G}$ be the $\sigma$-algebra generated by the sequence of record values $\{R_i\}_{i\in\N}$.

\begin{lemma}\label{prop_dist_t_extended}
	Assume that $\overline{F}(x)=x^{-\gamma}L(x)$, $x>D,$ is a Pareto-type distribution. Let $\delta\in(0,1)$, and also let $n\ge 1,m\ge 2$ be integers and $A>0$. Then:
	\begin{enumerate}[a)]
		\item \label{prop_dist_t_extended_a} The random variables $\mathbf{S}=(S_1,\ldots,S_n)$ are independent conditional on $\mathcal{G}$. Furthermore, for each $i=1,\ldots,n$, the random variable $S_i$ conditional on the event $\{R_i=r_i\}$, for some $r_i>A$, has a distribution $\dgeom(\delta^\gamma L( r_i)/L(\delta r_i))$.
		\item \label{prop_dist_t_extended_b} Let $i=1,\ldots,n$. Conditional on  $\mathcal{G}$ and $\{S_i=s_i \}$, for some integer $s_i\ge 1$, the random variables $V_i^1,\ldots,V_i^{s_i}$ are i.i.d. In addition, for each $j=1,\ldots,s_i$, the random variable $V_i^j$ conditioned on $\{R_i=r_i, S_i=s_i\}$, for some $r_i>A$, takes values in $\{0,\ldots,m-1\}$ with probability
		\begin{align*}
			\prob(V_i^j=v\mid R_i=r_i, S_i=s_i) &=\delta^{v\frac{\gamma}{m}}\frac{L(a^v\delta r_i) - \delta^{\frac{\gamma}{m}}L(a^{v+1}\delta r_i)}{L(\delta r_i) - \delta^\gamma L(r_i)},\quad v=0,\ldots,m-1,
		\end{align*}
		with $a = \delta^{-\frac{1}{m}}$.
	\end{enumerate}
\end{lemma}

\begin{lemma}\label{lemma_distribution_s_i_s_tilde_i} Assume that $\overline{F}(x)=x^{-\gamma}L(x)$, $x>D,$ is a Pareto-type distribution. Let $\delta\in(0,1)$, $m\ge 2$ be an integer, and $A>0$. Define $I:=\inf\{i\in\N:R^*_i>\delta^{-1}\}$. For all $i\ge I$, we have
	\begin{equation*}
		S_i\le S^*_i\quad\text{if }\ \dfrac{L(\delta R_i)}{L( R_i)} \le 1,\qquad		S_i\ge S^*_i\quad\text{if }\ \dfrac{L(\delta R_i)}{L( R_i)} > 1.
	\end{equation*}
	Moreover, almost surely,
	
	\begin{align*}
	\prob(&S_i=k, S_i^*=\ell\mid\mathcal{G}) \\
	&= \begin{cases}
		\delta^\gamma\displaystyle\binom{\ \ell\ }{\ k\ } \left(\dfrac{L(\delta R_i)}{L(R_i)}-\delta^\gamma\right)^k\left(1-\dfrac{L(\delta R_i)}{L(R_i)}\right)^{\ell-k}
		&\text{ if }\  \dfrac{L(\delta R_i)}{L( R_i)} \le 1 \text{ and } \ell \ge k\ge 0,\\
		\delta^\gamma \left(\dfrac{L(R_i)}{L(\delta R_i)}\right)^{k+1}\displaystyle\binom{\ k\ }{\ \ell\ }\ (1-\delta^\gamma)^\ell \left(\dfrac{L(\delta R_i)}{L(R_i)} -1\right)^{k-\ell} 
		&\text{ if }\  \dfrac{L(\delta R_i)}{L( R_i)} >1 \text{ and } k \ge\ell\ge 0.\\
	\end{cases}
	\end{align*}
	Consequently,
	$$
	\prob(S_i\ne S_i^*\mid\mathcal{G})\sim \delta^{-\gamma}\left|\frac{L(\delta R_i)}{L( R_i)}-1\right|\;\text{ a.s. as }\; i\to\infty.
	$$
\end{lemma}

\begin{lemma}\label{lemma_distribution_v_i_v_tilde_i}  Assume that $\overline{F}(x)=x^{-\gamma}L(x)$, $x>D,$ is a Pareto-type distribution. Let $\delta\in(0,1)$, $m\ge 2$ be an integer, and $A>0$. Define $I:=\inf\{i\in\N:R^*_i>\delta^{-1}\}$. For all $i\ge I$ and $s\ge1$, the pairs $(V_i^1, V_i^{*1}),\ldots,(V_i^{s}, V_i^{*s})$ are i.i.d., conditional on $\{S_i=S^*_i, S_i^* = s\}$  and on the $\sigma$-algebra $\mathcal{G}$, with common bivariate distribution
	\begin{align*}
		\prob(&V_i^1 = v, V_i^{*1} =v^*\mid S_i=S^*_i,S^*_i = s, \mathcal{G}) \\
		&= \left(\min\left\{\frac{L(\delta R_i)}{L(R_i)}, 1\right\}-\delta^\gamma\right)^{-1}\\
		&\quad\times \left(\min\left\{\delta^{v\frac{\gamma}{m}}\frac{L(a^v\delta R_i)}{L(R_i)}, \delta^{v^*\frac{\gamma}{m}}\right\}-\delta^{\frac{\gamma}{m}}\max\left\{\delta^{v\frac{\gamma}{m}}\frac{L(a^{v+1}\delta R_i)}{L(R_i)}, \delta^{v^*\frac{\gamma}{m}}\right\}\right)_+\text{ a.s.},
	\end{align*}
	for all $v,v^*=0,\ldots,m-1$. Consequently, almost surely,
	\begin{align*}
		\prob(&V_i^1 \neq V_i^{*1} \mid S_i=S_i^*, S_i^* = s, \mathcal{G})\\
		& \le  \frac{(2+\delta^\gamma)(1+\delta^{\frac{\gamma}{m}})}{1-\delta^{\frac{\gamma}{m}}}\left(\min\left\{\frac{L(\delta R_i)}{L(R_i)}, 1\right\}-\delta^\gamma\right)^{-1}\max_{0\le v\le m-1}\left|\frac{L(a^v\delta R_i)}{L(R_i)}- 1\right|.
	\end{align*}
\end{lemma}

\begin{lemma} \label{prop_bound_probability}Assume that $\overline{F}(x)=x^{-\gamma}L(x)$, $x>D,$ is a Pareto-type distribution. Let $\delta\in(0,1)$, $m\ge 2$ be an integer, and $A>0$. Define $I:=\inf\{i\in\N:R^*_i>\delta^{-1}\}$. For all $i\ge I$,
	$$
	\prob(S_i=S^*_i, S^*_i\ge 1, V_i^j\ne V_i^{*j}\text{ for some } j=1,\ldots,S^*_i\mid\mathcal{G})\le Y_i\;\text{ a.s.},
	$$
	where $\{Y_i\}_{i\in\N}$ is a sequence of random variables such that
	$$
	Y_i\sim \frac{(2+\delta^\gamma)(1+\delta^{\frac{\gamma}{m}})}{(1-\delta^{\frac{\gamma}{m}})\delta^\gamma}\max_{0\le v\le m-1}\left|\frac{L(a^v\delta R_i)}{L(R_i)}-1\right|\;\text{ a.s. as }\; i\to\infty.
	$$
\end{lemma}

\begin{lemma} \label{prop_LLN_Resnick} If $\overline{F}(x)=x^{-\gamma}L(x)$, $x>D,$ is a Pareto-type distribution,
	then
	\begin{equation*}
		\frac{\log{(R_i)}}{i}\toas \gamma^{-1}\; \text{ and }\;\frac{\log{\log{(R_i)}}}{\log{(i)}}\toas 1\;\text{ as }\; i\to\infty.
	\end{equation*}
\end{lemma}	
\begin{lemma}\label{lemma_convergence} If $\overline{F}(x)=x^{-\gamma}L(x)$, $x> D$, is differentiable and the function $L$ satisfies condition \eqref{eq_condition_L_3}, then
	$$
	\frac{L(tR_i)}{L(R_i)}-1=\mathrm{o}\ \bigg(\frac{1}{i(\log{(i)})^{2+\eta}}\bigg)\;\text{ a.s. as }\; i\to\infty,
	$$
	for all $t>0$.
\end{lemma}

The differences between the samples $\mathbf{T}$ and $\mathbf{T}^*$ are summarized in the following result.

\begin{prop}\label{theorem_comparison_t_tilde_t}
	Assume that $\overline{F}(x)=x^{-\gamma}L(x)$, $x> D$, is differentiable and that the function $L$ satisfies condition \eqref{eq_condition_L_3}. Let $\delta\in(0,1)$, $m\ge 2$ be an integer, and $A>0$. Then:
	\begin{align}
		\prob(S_i\neq S_i^*,i\ge 1\text{ i.o.}) &= 0,\label{eq_s}\\
		\prob(S_i^*=S_i, S_i^*\ge 1, V_i^j \neq V_i^{*j} \text{ for some } j=1,\ldots,S^*_i, i\ge 1\text{ i.o.}) &= 0,\label{eq_v}\\
		\prob(K_i\neq K_i^*, i\ge 1\text{ i.o.}) &= 0.\label{eq_k}
	\end{align}
	\begin{proof}
		We first prove \eqref{eq_s} and \eqref{eq_v}. From Lemma \ref{lemma_convergence}, and noting that 
		$$
		\sum_{i=2}^{\infty}\frac{1}{i(\log{(i)})^{2+\eta}}<\infty,
		$$
		we obtain 
		\begin{equation*}
			\sum_{i=1}^\infty \left|\frac{L(tR_i)}{L(R_i)}-1\right|<\infty\;\text{ a.s. for all } t>0.
		\end{equation*}
		This convergence, together with Lemma \ref{lemma_distribution_s_i_s_tilde_i} and \ref{prop_bound_probability}, yields 
		\begin{align*}
			\sum_{i=1}^\infty\prob(S_i\neq S_i^*\mid \mathcal{G}) &<\infty\;\text{ a.s.},
			\\
			\sum_{i=1}^\infty\prob(S^*_i=S_i, S^*_i\ge 1, V_i^j \neq V_i^{*j} \text{ for some } j=1,\ldots,S^*_i\mid\mathcal{G}) &< \infty\label{eq_v_2}\;\text{ a.s.}
		\end{align*}
		This implies 	
		\eqref{eq_s} and \eqref{eq_v} by the first conditional Borel-Cantelli Lemma (see \cite{Chen17}).
		
		Regarding \eqref{eq_k}, we define
		\begin{equation*}
			E_i:= H(R_i) - H(R_{i-1}) = \gamma\log{\bigg(\frac{R_i}{R_{i-1}}\bigg)} - \log{\bigg(\frac{L(R_i)}{L(R_{i-1})}\bigg)}, \quad i\in\N,
		\end{equation*}
		with $H(x):= -\log{\overline{F}(x)} + \log{\overline{F}(A)}$, for $x>A$, being the cumulative hazard function of the truncated distribution $\overline{F}(\cdot)/\overline{F}(A)$. Note that, as shown in \cite{Arnold98}, the random variables $\{E_i\}_{i\in\N}$ form an i.i.d.~exponential sequence with rate $1$. Since, almost surely, for all $i\in\N$,
		\begin{equation*}
			K_i = \left\lfloor\log_a{\left(\frac{R_i}{R_{i-1}}\right)}\right\rfloor\;\text{ and }\; K_i^* =  \bigg\lfloor\log_a{\bigg(\frac{R_i^*}{R_{i-1}^*}\bigg)}\bigg\rfloor= \left\lfloor\frac{E_i}{\gamma\log{(a)}}\right\rfloor
		\end{equation*}
		we have that \eqref{eq_k} is equivalent to the following two conditions:
		\begin{align}
			&\prob\bigg(\log_a{\bigg(\frac{R_{i}}{R_{i-1}}\bigg)} < \bigg\lfloor\frac{E_{i}}{\gamma\log{(a)}}\bigg\rfloor, i\ge 1\text{ i.o.}\bigg)=0\;\text{ and}\label{eq_c}\\
			&\prob\bigg(\log_a{\bigg(\frac{R_{i}}{R_{i-1}}\bigg)} \ge \bigg\lfloor\frac{E_{i}}{\gamma\log{(a)}}\bigg\rfloor + 1,i\ge 1\text{ i.o.}\bigg)=0.\label{eq_d}
		\end{align}
		To prove \eqref{eq_c}, we use the following identity, which is a consequence of an application of the Mean Value Theorem analogous to that in the proof of Lemma \ref{lemma_convergence}. Given $i\in\N$, we have
		\begin{equation}
			\log{\bigg(\frac{L(R_{i})}{L(R_{i-1})}\bigg)} = \frac{\e^{C_i}L'(\e^{C_i})}{L(\e^{C_i})}(\log{(R_i)}-\log{(R_{i-1})})
			=\frac{\e^{C_i}L'(\e^{C_i})}{L(\e^{C_i})}E_{i}\bigg(\gamma -  \frac{\e^{C_i}L'(\e^{C_i})}{L(\e^{C_i})}\bigg)^{-1},\label{eq_yy}
		\end{equation}
		where $C_i\in(\log{(R_{i-1})}, \log{(R_{i})})$ with probability $1$. Moreover, observe that, for all $i\ge 1$,
		$$
		\left\{\log_a{\bigg(\frac{R_{i}}{R_{i-1}}\bigg)} < \bigg\lfloor\frac{E_{i}}{\gamma\log{(a)}}\bigg\rfloor\right\} = \left\{-\frac{1}{\gamma}\log_a{\bigg(\frac{L(R_{i})}{L(R_{i-1})}\bigg)} > \bigg\{\frac{E_{i}}{\gamma\log{(a)}}\bigg\}\right\}.
		$$
		Thus, from  \eqref{eq_yy}, it follows that the left-hand side of \eqref{eq_c} is equal to
		{\allowdisplaybreaks
			\begin{align}
				\prob&\bigg(-\frac{1}{\gamma\log{(a)}}\frac{\e^{C_i}L'(\e^{C_i})}{L(\e^{C_i})}\ \bigg(\gamma -  \frac{\e^{C_i}L'(\e^{C_i})}{L(\e^{C_i})}\bigg)^{-1}E_{i} >\bigg\{\frac{E_{i}}{\gamma\log{(a)}}\bigg\}, i\ge 1\text{ i.o.}\bigg)\nonumber\\
				&\le\prob\bigg(\frac{1}{\gamma\log{(a)}}\bigg|\frac{\e^{C_i}L'(\e^{C_i})}{L(\e^{C_i})}\bigg|\ \bigg|\gamma -  \frac{\e^{C_i}L'(\e^{C_i})}{L(\e^{C_i})}\bigg|^{-1}E_{i} >\bigg\{\frac{E_{i}}{\gamma\log{(a)}}\bigg\}, i\ge 1\text{ i.o.}\bigg)\nonumber\\
				& = \prob\bigg(\frac{1}{\gamma\log{(a)}}\bigg|\frac{\e^{C_i}L'(\e^{C_i})}{L(\e^{C_i})}C_i(\log{(C_i)})^{2+\eta}\bigg|\nonumber\\
				&\qquad\times \bigg|\frac{i(\log{(i)})^{1+\eta}}{C_i(\log{(C_i)})^{1+\eta}}\bigg|\ \bigg|\gamma -  \frac{\e^{C_i}L'(\e^{C_i})}{L(\e^{C_i})}\bigg|^{-1}\frac{E_{i}}{\log{(i)}}\ \bigg|\frac{\log{(i)}}{\log{(C_i)}}\bigg|\nonumber\\
				&\qquad> i(\log{(i)})^{1+\eta} \bigg\{\frac{E_{i}}{\gamma\log{(a)}}\bigg\}, i\ge 1\text{ i.o.}\bigg).\label{eq_313}
		\end{align}}
		We claim that the probability  in \eqref{eq_313} is equal to $0$. Indeed, note first that $C_i\sim \log{(R_i)}$ and $\log{(C_i)}\sim \log{\log{(R_i)}}$ a.s. as $i\to\infty$ (see Lemma \ref{prop_LLN_Resnick} above), yielding
		\begin{equation}
			\limsup_{i\to\infty} \bigg|\frac{i(\log{(i)})^{1+\eta}}{C_i(\log{(C_i)})^{1+\eta}}\bigg|\ \bigg|\gamma -  \frac{\e^{C_i}L'(\e^{C_i})}{L(\e^{C_i})}\bigg|^{-1}\frac{E_{i}}{\log{(i)}}\ \bigg|\frac{\log{(i)}}{\log{(C_i)}}\bigg|  = 1\;\text{ a.s.}\label{eq_11}
		\end{equation}
		by Lemma \ref{prop_LLN_Resnick}, Lemma \ref{prop_exponential_appendix} in the Appendix, and condition \eqref{eq_condition_L_3}. Secondly, observe that
		\begin{equation}
			\lim_{i\to\infty}\frac{\e^{C_i}L'(\e^{C_i})}{L(\e^{C_i})}C_i(\log{(C_i)})^{2+\eta}=0\;\text{ a.s.}\label{eq_12}
		\end{equation}
		due to condition \eqref{eq_condition_L_3}. Lastly, note that
		\begin{equation}
			\lim_{i\to\infty} i(\log{(i)})^{1+\eta}\bigg\{\frac{E_{i}}{\gamma\log{(a)}}\bigg\}=\infty\;\text{ a.s.}\label{eq_13}
		\end{equation}
		as a consequence of Lemma \ref{lemma_liminf_exponential} in the Appendix and the fact that
		$$
		\sum_{i=2}^\infty \frac{1}{i(\log{(i)})^{1+\eta}} <\infty.
		$$
		Equations \eqref{eq_11}, \eqref{eq_12}, and \eqref{eq_13} imply that the inequality given in \eqref{eq_313} can occur only finitely many times with probability 1. Equation \eqref{eq_d} can be proved analogously.
	\end{proof}
\end{prop}

\begin{remark}We point out that \eqref{eq_s} and \eqref{eq_v} hold under the weaker condition
	\begin{equation*}
		\frac{xL'(x)}{L(x)} = \mathrm{o}\bigg(\frac{1}{\log{(x)}(\log{\log{(x)}})^{1+\eta}}\bigg)\;\text{ as }\; x\to\infty,
	\end{equation*}
	for some $\eta>0$. Indeed, under this condition it can be shown that, for all $t>0$,
	$|L(tR_i)/L(R_i)-1|=\mathrm{o}(i^{-1}(\log{i})^{-(1+\eta)})$ a.s.~as $i\to\infty$ by following the proof of Lemma \ref{lemma_convergence}. Since the series $\sum_{i\ge 2} i^{-1}(\log i)^{-(1+\eta)}$ converges, the argument used in the proof of \eqref{eq_s} and \eqref{eq_v} in Proposition \ref{theorem_comparison_t_tilde_t} can be applied without modification.
\end{remark}

The main result of this section, which states the asymptotic behavior of $\hat{\gamma}_{\delta,m,n}^A$ for Pareto-type distributions, is the following.

\begin{theorem}\label{theorem_pareto_type}
	Assume that $\overline{F}(x)=x^{-\gamma}L(x)$, $x> D$, is differentiable and that the function $L$ satisfies condition \eqref{eq_condition_L_3}. Let $\delta\in(0,1)$, $m\ge 2$ be an integer, and $A>0$. Then,
	\begin{enumerate}[a)]
		\item $$
		\hat{\gamma}_{\delta,m,n}^A\toas \gamma\;\text{ as }\; n\to\infty,
		$$
		\item $$
		\sqrt{n}(\hat{\gamma}_{\delta, m,n}^A-\gamma)\Rightarrow \sigma_{\delta,m}Z\;\text{ as }\;n\to\infty,
		$$
	\end{enumerate}
	with $\sigma_{ \delta, m}$  as in  \eqref{asymptotic_variance} and $Z$ has a standard normal distribution.
\end{theorem}
\begin{proof}
	By Proposition \ref{theorem_comparison_t_tilde_t}, the random index 
	$$M:=\inf\{i\in\N: \text{for all } \ell\ge i, K_\ell=K_\ell^*, S_\ell=S_\ell^*, \text{ and }V_\ell^j=V_\ell^{*j}\text{ for } j=1,\ldots,S_\ell\}$$ 
	is almost surely well defined. Thus, we have that the samples $\mathbf{T}$ and $\mathbf{T}^*$ only differ, at most, in $M-1$ components with probability $1$, with $\mathbf{T}^*$ being a sample obtained from Pareto observations. This yields $|\hat{\gamma}^A_{\delta,m,n}-\hat{\gamma}^*_{\delta,m,n}| = \mathrm{O}(n^{-1})$ a.s. as $n\to\infty$ by Proposition \ref{prop_AB} and Remark \ref{remark_threshold_A}. The result then follows from the asymptotic properties of $\hat{\gamma}^*_{\delta,m,n}$ in \eqref{convgamma*}.
\end{proof}

To conclude this section, we note that the MLE of the asymptotic standard deviation $\sigma_{ \delta, m}$, which is given by
\begin{equation*}
	\hat{\sigma}_{\delta,m,n}^A = \sigma_{\delta,m} (\hat{\gamma}_{\delta,m,n}^A)= \frac{m(1-\delta^{\frac{\hat\gamma}{m}})\delta^{\frac{\hat\gamma}{2}(1 - \frac{1}{m})}}{-\log{(\delta)}},
\end{equation*}
is strongly consistent as $n\to\infty$ as a consequence of Theorem \ref{theorem_pareto_type} a). Thus, Slutsky's Theorem, together with Theorem \ref{theorem_pareto_type} b), enables us to conclude that
\begin{equation}
	\left[\hat{\gamma}_{\delta,m,n}^A - z_{1-\frac{\alpha}{2}}\frac{\hat{\sigma}_{\delta,m,n}^A}{\sqrt{n}}, \hat{\gamma}_{\delta,m,n}^A + z_{1-\frac{\alpha}{2}}\frac{\hat{\sigma}_{\delta,m,n}^A}{\sqrt{n}}\right]\label{confidence_interval_gamma}
\end{equation}
is an asymptotic confidence interval for $\gamma$ with level $1-\alpha$, $\alpha\in(0,1)$. The number $z_\beta$, with $\beta\in(0,1)$, denotes the $\beta$-quantile of the standard normal distribution.

\begin{remark}
	It should be noted that our estimator in \eqref{eq_estimator_gamma} was derived under the assumption that all geometric near-records associated with the first $n$ records are observed. This implicitly requires that the $(n+1)$-th record has already occurred, since only then can we be certain that no additional geometric near-records associated with the $n$-th record will appear.
	However, when the estimator is applied to a set of real data with $n$ records, it is unclear if all the geometric near-records associated with the $n$-th record have been observed. In other words, the last block of geometric near-records may be incomplete. To account for this situation, the likelihood for Pareto samples in \eqref{eq_2} must be slightly modified by replacing $\prob(S_n=s_n)=(1-\delta^\gamma)^{s_n}\delta^{\gamma}$ with $\prob(S_n\ge s_n) = (1-\delta^\gamma)^{s_n}$. The likelihood then becomes
	\begin{equation*}
		\mathcal{L}(\mathbf{t}; \gamma) =
		(\delta^{\gamma})^{n-1 + \frac{1}{m}\sum\limits_{i=1}^n\sum\limits_{j=1}^{s_i}v_i^j+\frac{1}{m}\sum\limits_{i=1}^{n}k_i}(1-\delta^{\frac{\gamma}{m}})^{n +\sum\limits_{i=1}^n s_i},
	\end{equation*}
 and the MLE is given by
	\begin{equation}
		\hat{\gamma}_{\delta, m, n}^A =
		m\log_\delta{(\hat{\beta}_{\delta,m,n}^A)}, \label{eq_estimator_practical}
	\end{equation}
	where
	\begin{equation}
		\hat{\beta}_{\delta,m,n}^A =
		\dfrac{m(n-1) + \sum\limits_{i=1}^n\sum\limits_{j=1}^{S_i}V_i^j + \sum\limits_{i=1}^{n}K_i}
		{(m+1)n - m + \sum\limits_{i=1}^n\sum\limits_{j=1}^{S_i}V_i^j + \sum\limits_{i=1}^{n}K_i+\sum\limits_{i=1}^n S_i}. \label{eq_estimator_beta_practical}
	\end{equation}
	Thus, when applying the estimator to a real data set, unless it is known that the next record has already occurred (and therefore that all geometric near-records associated with the $n$-th record have been observed), estimator \eqref{eq_estimator_practical} should be used. Owing to the similarity between the estimators in \eqref{eq_estimator_gamma} and \eqref{eq_estimator_practical}, the asymptotic properties established for the former can also be derived for the latter.
\end{remark}

\subsection{Proofs of Lemmas \ref{prop_dist_t_extended}-\ref{lemma_convergence}.}\label{proofLemmas}
\begin{proof}[Proof of Lemma \ref{prop_dist_t_extended}]
	It suffices to follow the proof of parts $a)$ and $b)$ of Proposition \ref{prop_distribution_t} conditioning on $\mathcal{G}$.
\end{proof}

\begin{proof}[Proof of Lemma \ref{lemma_distribution_s_i_s_tilde_i}]
	An observation $X_j$, with $j\in\{L_i+1,\ldots,L_{i+1}-1\}$, is a geometric near-record in the sequence $\{X_n\}_{n\in\N}$ or in the sequence $\{X^*_n\}_{n\in\N}$ if and only if the following events occur, respectively: 
	\begin{align}
		\{X_j\in (\delta R_i, R_i]\} = \{\varphi(X_j) \in(\varphi(\delta R_i), \varphi(R_i)]\},\label{suceso1}\\
		\{X^*_j \in(\delta R^*_i, R^*_i]\}=\{\varphi(X_j) \in(\delta\varphi(R_i), \varphi(R_i)]\}, \label{suceso2}
	\end{align}
where the equalities above are to be understood in the almost sure sense. Hence, we have $S_i\le S^*_i$ if $L(\delta R_i)/L(R_i)\le 1$ and $S_i\ge S^*_i$ otherwise.
	
	Assume now that $L(\delta R_i)/L(R_i)\le 1$ and let $\ell\ge k\ge 0$ be integers. Note that $\prob(S_i=k, S^*_i=\ell\mid\mathcal{G}) = \prob(S_i=k\mid S^*_i=\ell,\mathcal{G})\prob(S^*_i=\ell\mid\mathcal{G})$ with probability $1$. Let us compute $\prob(S_i=k\mid S^*_i=\ell,\mathcal{G})$. Conditioned on $\mathcal{G}$ and $\{ S^*_i=\ell\}$, the occurrence of $\{S_i=k\}$ is equivalent to the occurrence of the event \eqref{suceso1} $k$ times and the occurrence of its complementary event with respect to \eqref{suceso2}, which is $\{\varphi(X_j)\in(\delta\varphi(R_i),\varphi(\delta R_i)] \}$, $\ell-k$ times. Thus, the distribution of $S_i$ conditioned on $\{ S^*_i=\ell,\mathcal{G}\}$ is binomial with $\ell$ trials and probability of success
	$$
	\prob\left(\varphi(X_j)\in \left.\left(\varphi(\delta R_i), \varphi(R_i)\right]\right.\ \big\vert\ \varphi(X_j)\in \left.\left(\delta \varphi(R_i), \varphi(R_i)\right]\right.,\mathcal{G}\right) = \frac{1}{1-\delta^{\gamma}}\left(\frac{L(\delta R_i)}{L(R_i)}-\delta^{\gamma}\right)\;\text{ a.s.},
	$$
	yielding
	$$
	\prob(S_i=k\mid S^*_i=\ell,\mathcal{G}) = 
	\binom{\ \ell\ }{\ k\ }\ \frac{1}{(1-\delta^{\gamma})^\ell}\left(\frac{L(\delta R_i)}{L(R_i)}-\delta^{\gamma}\right)^k \left(1-\frac{L(\delta R_i)}{L(R_i)}\right)^{\ell-k}\;\text{ a.s.}
	$$
	Now, using $\prob(S^*_i=\ell\mid\mathcal{G})=\prob(S^*_i=\ell)=(1-\delta^\gamma)^\ell\delta^\gamma$ a.s.~for $i\ge I$ by Proposition \ref{prop_distribution_t} a), we obtain the expression for $\prob(S_i=k, S^*_i=\ell\mid\mathcal{G})$. The case $L(\delta R_i)/L(R_i)> 1$ can be treated analogously.
	
	Regarding the last claim, it is easy to see that, almost surely,
	{\allowdisplaybreaks
		\begin{align*}
			\prob(S_i\ne S^*_i\mid \mathcal{G}) &= \prob(S_i < S^*_i \mid\mathcal{G})\mathbf{1}{\left\{\frac{L(\delta R_i)}{L(R_i)}\le 1\right\}}+ \prob(S_i > S^*_i \mid\mathcal{G})\mathbf{1}{\left\{\frac{L(\delta R_i)}{L(R_i)} > 1\right\}}\\
			&=\mathbf{1}{\left\{\frac{L(\delta R_i)}{L(R_i)}\le 1\right\}}\sum_{\ell=1}^\infty\sum_{k=0}^{\ell - 1} \prob(S_i=k,S^*_i=\ell\mid\mathcal{G})\\
			&\qquad+\mathbf{1}{\left\{\frac{L(\delta R_i)}{L(R_i)} > 1\right\}}\sum_{k=1}^\infty\sum_{\ell=0}^{k-1} \prob(S_i=k,S^*_i=\ell\mid\mathcal{G})\\
			&= \mathbf{1}{\left\{\frac{L(\delta R_i)}{L(R_i)}\le 1\right\}}\left(1-\frac{L(\delta R_i)}{L(R_i)}\right)\left(1-\frac{L(\delta R_i)}{L(R_i)}+\delta^\gamma\right)^{-1}\\
			&\qquad+ \mathbf{1}{\left\{\frac{L(\delta R_i)}{L(R_i)}> 1\right\}}\left(1-\frac{L(R_i)}{L(\delta R_i)}\right)\left(1-\frac{L(R_i)}{L(\delta R_i)}+\delta^\gamma\frac{L(R_i)}{L(\delta R_i)}\right)^{-1}\\
			&\sim \delta^{-\gamma}\left|\frac{L(\delta R_i)}{L(R_i)}-1\right|\;\text{ a.s. as }\; i\to\infty.
	\end{align*}}
	\end{proof}

	\begin{proof}[Proof of Lemma \ref{lemma_distribution_v_i_v_tilde_i}]
		The conditional independence and identical distribution of the pairs $(V_i^1, V_i^{*1}),\ldots,(V_i^s, V_i^{*s})$, for $i\ge I$, are straightforward. Next, observe that, given $S_i=S^*_i$, every geometric near-record $X_j$ associated with $R_i$ satisfies $\varphi(X_j)\in(\max\{\varphi(\delta R_i), \delta \varphi(R_i)\}, \varphi(R_i)]$ a.s. Moreover, note that the event $\{V_i^1=v,V_i^{*1}=v^*\}$ is almost surely equal to $\{\varphi(X_j)\in(\varphi(a^v\delta R_i), \varphi(a^{v+1}\delta R_i)],\varphi(X_j)\in(a^{v^*}\delta\varphi(R_i),a^{v^*+1}\delta\varphi(R_i)]\}$ for some geometric near-record $X_j$, that is,
		\begin{align*}
			\varphi(X_j)\in\left(\max\{\varphi(a^v\delta R_i), a^{v^*}\delta\varphi(R_i)\}, \min\{\varphi(a^{v+1}\delta R_i), a^{v^*+1}\delta\varphi( R_i)\}\right]\;\text{ a.s.}
		\end{align*}
		Thus, with probability $1$,
		\begin{align}
			\prob(&V_i^1 = v, V_i^{*1} =v^*\mid S_i=S^*_i,S^*_i = s, \mathcal{G})\nonumber\\
			&= \frac{\prob\left(\varphi(X_j)\in\left.\left(\max\{\varphi(a^v\delta R_i), a^{v^*}\delta\varphi(R_i)\}, \min\{\varphi(a^{v+1}\delta R_i), a^{v^*+1}\delta\varphi( R_i)\}\right.\right] \ \big\vert\ \mathcal{G}\right)}{\prob\left(\varphi(X_j)\in(\max\{\varphi(\delta R_i), \delta \varphi(R_i)\}, \varphi(R_i)]\ \big\vert\ \mathcal{G}\right)}\nonumber\\
			&= \left(\min\left\{\frac{L(\delta R_i)}{L(R_i)}, 1\right\}-\delta^\gamma\right)^{-1}\nonumber\\
			&\qquad\times \left(\min\left\{\delta^{v\frac{\gamma}{m}}\frac{L(a^v\delta R_i)}{L(R_i)}, \delta^{v^*\frac{\gamma}{m}}\right\}-\delta^{\frac{\gamma}{m}}\max\left\{\delta^{v\frac{\gamma}{m}}\frac{L(a^{v+1}\delta R_i)}{L(R_i)}, \delta^{v^*\frac{\gamma}{m}}\right\}\right)_+\label{eq_bivariate}.
		\end{align}
		Using the joint distribution \eqref{eq_bivariate}, we obtain that, almost surely,
		{\allowdisplaybreaks
			\begin{align}
				\prob(V_i^1 &\neq V_i^{*1} \mid S_i=S^*_i, S^*_i = s, \mathcal{G})\nonumber\\
				& = 1 - \sum_{v=0}^{m-1} \prob(V_i^1 = v, V_i^{*1} =v\mid S_i=S^*_i,S^*_i = s, \mathcal{G})\nonumber\\
				&=\sum_{v=0}^{m-1} \left[\frac{\delta^{v\frac{\gamma}{m}}(1-\delta^{\frac{\gamma}{m}})}{1-\delta^\gamma}-\left(\min\left\{\frac{L(\delta R_i)}{L(R_i)}, 1\right\}-\delta^\gamma\right)^{-1}\nonumber\right.\\
				&\left.\qquad\qquad\times\left(\min\left\{\delta^{v\frac{\gamma}{m}}\frac{L(a^v\delta R_i)}{L(R_i)}, \delta^{v\frac{\gamma}{m}}\right\}-\delta^{\frac{\gamma}{m}}\max\left\{\delta^{v\frac{\gamma}{m}}\frac{L(a^{v+1}\delta R_i)}{L(R_i)}, \delta^{v\frac{\gamma}{m}}\right\}\right)_+\right]\nonumber\\
				&\le\sum_{v=0}^{m-1} \left[\frac{\delta^{v\frac{\gamma}{m}}(1-\delta^{\frac{\gamma}{m}})}{1-\delta^\gamma}-\left(\min\left\{\frac{L(\delta R_i)}{L(R_i)}, 1\right\}-\delta^\gamma\right)^{-1}\nonumber\right.\\
				&\left.\qquad\qquad\times\left(\min\left\{\delta^{v\frac{\gamma}{m}}\frac{L(a^v\delta R_i)}{L(R_i)}, \delta^{v\frac{\gamma}{m}}\right\}-\delta^{\frac{\gamma}{m}}\max\left\{\delta^{v\frac{\gamma}{m}}\frac{L(a^{v+1}\delta R_i)}{L(R_i)}, \delta^{v\frac{\gamma}{m}}\right\}\right)_{\phantom{+}\!\!\!\!}\right]\nonumber\\ 
				&= \frac{1}{1-\delta^\gamma}\left(\min\left\{\frac{L(\delta R_i)}{L(R_i)}, 1\right\}-\delta^\gamma\right)^{-1}\sum_{v=0}^{m-1}\delta^{v\frac{\gamma}{m}}(f_1(v)-\delta^{\frac{\gamma}{m}}f_2(v)),\label{eq_iop}
		\end{align}}
		where the functions $f_1(v)$ and $f_2(v)$ are defined as:
		\begin{align*}
			f_1(v)&:= \min\left\{\frac{L(\delta R_i)}{L(R_i)}, 1\right\} - \min\left\{\frac{L(a^v\delta R_i)}{L(R_i)}, 1\right\}+ \delta^\gamma\min\left\{\frac{L(a^v\delta R_i)}{L(R_i)}, 1\right\} - \delta^\gamma\\
			&= \min\left\{\frac{L(\delta R_i)}{L(R_i)}-1, 0\right\} - \min\left\{\frac{L(a^v\delta R_i)}{L(R_i)}-1, 0\right\}+ \delta^\gamma\min\left\{\frac{L(a^v\delta R_i)}{L(R_i)}-1, 0\right\},\\
			f_2(v)&:= \min\left\{\frac{L(\delta R_i)}{L(R_i)}, 1\right\} - \max\left\{\frac{L(a^{v+1}\delta R_i)}{L(R_i)}, 1\right\}+ \delta^\gamma\max\left\{\frac{L(a^{v+1}\delta R_i)}{L(R_i)}, 1\right\}- \delta^\gamma\\
			&=\min\left\{\frac{L(\delta R_i)}{L(R_i)}-1, 0\right\} - \max\left\{\frac{L(a^{v+1}\delta R_i)}{L(R_i)}-1, 0\right\}+ \delta^\gamma\max\left\{\frac{L(a^{v+1}\delta R_i)}{L(R_i)}-1, 0\right\},
		\end{align*}
		for all $v=0,\ldots,m-1$. Note that  $f_1(v)$ and $f_2(v)$ are bounded above by
		\begin{align*}
			\left|\frac{L(\delta R_i)}{L(R_i)}-1\right|&+\max_{0\le v\le m-1}\left|\frac{L(a^v\delta R_i)}{L(R_i)}-1\right|+\delta^\gamma\max_{0\le v\le m-1}\left|\frac{L(a^v\delta R_i)}{L(R_i)}-1\right|\\
			&\le(2+\delta^\gamma)\max_{0\le v\le m-1}\left|\frac{L(a^v\delta R_i)}{L(R_i)}-1\right|.
		\end{align*}
		Therefore, we have that the expression  in the last line of \eqref{eq_iop} is smaller than or equal to
		\begin{align*}
			\frac{2+\delta^\gamma}{1-\delta^\gamma}&\left(\min\left\{\frac{L(\delta R_i)}{L(R_i)}, 1\right\}-\delta^\gamma\right)^{-1}\max_{0\le v\le m-1}\left|\frac{L(a^v\delta R_i)}{L(R_i)}-1\right|(1+\delta^{\frac{\gamma}{m}})\sum_{v=0}^{m-1}\delta^{v\frac{\gamma}{m}}\\
			&= \frac{(2+\delta^\gamma)(1+\delta^{\frac{\gamma}{m}})}{1-\delta^{\frac{\gamma}{m}}}\left(\min\left\{\frac{L(\delta R_i)}{L(R_i)}, 1\right\}-\delta^\gamma\right)^{-1}\max_{0\le v\le m-1}\left|\frac{L(a^v\delta R_i)}{L(R_i)}-1\right|,
		\end{align*}
		thereby completing the proof.
	\end{proof}
	
	\begin{proof}[Proof of Lemma \ref{prop_bound_probability}]
		By Lemmas \ref{lemma_distribution_s_i_s_tilde_i} and \ref{lemma_distribution_v_i_v_tilde_i}, we have, for all $i\ge I$,
		{\allowdisplaybreaks
			\begin{align*}
				\prob(&S_i=S^*_i, S^*_i\ge 1, V_i^j\ne V_i^{*j}\text{ for some } j=1,\ldots,S^*_i\mid\mathcal{G})\\
				&=\sum_{\ell=1}^\infty \prob(S_i=\ell, S^*_i=\ell\mid\mathcal{G})\ \prob(V_i^j\ne V_i^{*j}\text{ for some } j=1,\ldots,\ell\mid S_i=S^*_i, S^*_i=\ell,\mathcal{G})\\
				&\le \sum_{\ell=1}^\infty \ell\ \prob(S_i=\ell, S^*_i=\ell\mid\mathcal{G})\ \prob(V_i^1\ne V_i^{*1}\mid S_i=S^*_i, S^*_i=\ell,\mathcal{G})\\
				&=\delta^\gamma\frac{(2+\delta^\gamma)(1+\delta^{\frac{\gamma}{m}})}{1-\delta^{\frac{\gamma}{m}}}\left(\min\left\{\frac{L(\delta R_i)}{L(R_i)}, 1\right\}-\delta^\gamma\right)^{-1}\max_{0\le v\le m-1}\left|\frac{L(a^v\delta R_i)}{L(R_i)}-1\right|\\
				&\qquad\times\left(\mathbf{1}{\left\{\frac{L(\delta R_i)}{L(R_i)}>1\right\}}\frac{L(R_i)}{L(\delta R_i)}\sum_{\ell=1}^\infty \ell \bigg((1-\delta^\gamma)\frac{L(R_i)}{L(\delta R_i)}\bigg)^\ell\right.\\
				&\left.\qquad\qquad+\mathbf{1}{\left\{\frac{L(\delta R_i)}{L(R_i)}\le 1\right\}}\sum_{\ell=1}^\infty \ell \bigg(\frac{L(\delta R_i)}{L(R_i)}-\delta^\gamma\bigg)^\ell\right)\\
				&\sim \frac{(2+\delta^\gamma)(1+\delta^{\frac{\gamma}{m}})}{(1-\delta^{\frac{\gamma}{m}})\delta^\gamma}\max_{0\le v\le m-1}\left|\frac{L(a^v\delta R_i)}{L(R_i)}-1\right|\;\text{ a.s. as }\; i\to\infty.
			\end{align*}
		}
	\end{proof}
	
	\begin{proof}[Proof of Lemma \ref{prop_LLN_Resnick}]
		Observe that
		\begin{align*}
			\frac{\log{(R_i^*)}}{i} &= 	\frac{\log{(R_i)}}{i}\left(1 - \frac{\log{L(R_i)}}{\gamma\log{(R_i)}}\right)\sim \frac{\log{(R_i)}}{i} \;\text{ a.s. as }\; i\to\infty,
		\end{align*}
		since the slow variation of $L$ implies that $\log{L(x)} /\log{(x)}\to 0$ as $x\to\infty$ (see \cite{Bingham89}, Proposition 1.3.6, p. 16). Thus, it suffices to note that $\{\log(R_i^*/\varphi(A))\}_{i\in\N}$ behaves as a sequence of records drawn from an exponential distribution with parameter $\gamma$, which corresponds to the arrival times of a homogeneous Poisson process with rate $\gamma$ (see \cite{Shorrock72}). Hence,
		$$
		\frac{\log{(R_i^*)}}{i} = \frac{\log{(R_i^*/\varphi(A))}}{i} + \frac{\log{\varphi(A)}}{i} \toas \gamma^{-1}\;\text{ as }\; i\to\infty,
		$$
		which completes the proof of the first statement. The second assertion follows immediately.
	\end{proof}
	
	\begin{proof}[Proof of Lemma \ref{lemma_convergence}]
		Let $t> 0$. Since $L$ is slowly varying,
		$$
		\frac{L(tR_i)}{L(R_i)}-1\sim \log{\bigg(\frac{L(tR_i)}{L(R_i)}\bigg)}\;\text{ a.s. as } i\to\infty,
		$$
		so we will prove the result for the latter sequence. Let $i\in\N$. Note that
		$$
		\left|\log{\bigg(\frac{L(tR_i)}{L(R_i)}\bigg)}\right| = |h(\log{(tR_{i})})- h(\log{(R_{i})})|,
		$$
		where $h(x):= \log L(\e^x)$, $x\in\R$, is a differentiable function such that $h'(x) = \e^x L'(\e^x)/L(\e^x)$ for each $x\in\R$. Therefore, by the Mean Value Theorem, there exists, with probability $1$, some $C_i$ between $\log{(tR_{i})}$ and $\log{(R_{i})}$ such that
		\begin{align*}
			\left|\log{\bigg(\frac{L(tR_i)}{L(R_i)}\bigg)}\right|& i(\log{(i)})^{2+\eta} = \left|\frac{\e^{C_i}L'(\e^{C_i})}{L(\e^{C_i})}(\log{(tR_i)}-\log{(R_i)})\right|\ i(\log{(i)})^{2+\eta}\\
			&= |\log{(t)}|\left|\frac{\e^{C_i}L'(\e^{C_i})}{L(\e^{C_i})}\log{(\e^{C_i})}(\log{\log{(\e^{C_i})}})^{2+\eta}\right|\left|\frac{i(\log{(i)})^{2+\eta}}{C_i(\log{(C_i)})^{2+\eta}}\right|.
		\end{align*}
		Letting $i\to\infty$, we have
		$$
		\frac{i(\log{(i)})^{2+\eta}}{C_i(\log{(C_i)})^{2+\eta}}\toas \gamma,
		$$	
		by Lemma \ref{prop_LLN_Resnick}, together with the fact that $C_i\sim \log{(R_i)}$ and $\log{(C_i)}\sim \log{\log{(R_i)}}$ almost surely. Moreover, condition \eqref{eq_condition_L_3} implies
		$$
		\frac{\e^{C_i}L'(\e^{C_i})}{L(\e^{C_i})}\log{(\e^{C_i})}(\log{\log{(\e^{C_i})}})^{2+\eta}\toas 0
		$$
		and the result follows.
	\end{proof}

\section{Finite-sample performance and application} \label{section_simulations}

In this section, we assess the performance of our estimator. First, we use Monte Carlo simulations to analyze its behavior under different settings and compare it with some classical estimators. Then, we illustrate its practical use through a real financial data set.

The estimators used for comparison are Hill's estimator \cite{Hill75}, which is based on extreme order statistics, and Berred's estimators \cite{Berred92}, based on record values. Hill's estimator is defined as
\begin{equation*}
	\hat{\gamma}_{k}^{\text{H}} := \left\{\frac{1}{k}\sum_{i=1}^k \log{\left(\frac{X_{(i)}}{X_{(k+1)}}\right)}\right\}^{-1},
\end{equation*}
where $X_{(1)}\ge X_{(2)}\ge \cdots \ge X_{(k+1)}$ are the $k+1$ largest order statistics from the sample, and Berred's estimators are
\begin{align*}
	\hat{\gamma}_{\ell, n}^{\text{B1}} &:= \left\{\frac{1}{\ell}\log{\left(\frac{R_{n}}{R_{n-\ell}}\right)}\right\}^{-1},\\
	\hat{\gamma}_{\ell, n}^{\text{B2}} &:= \left\{\frac{1}{n\ell - \frac{1}{2}\ell(\ell-1)}\sum_{i=1}^\ell \log{(R_{n-i+1})}\right\}^{-1},
\end{align*}
where $\{R_n\}_{n\in\N}$ denotes the sequence of records and $1\le \ell<n$.

We compare our estimator with Berred's estimators in terms of the MSE for a fixed number of records, $n$. The comparison with Hill's estimator is more involved, since its performance is highly sensitive to the choice of $k$. To make the comparison fair, we also consider the number of the required observations as a performance measure, which is particularly relevant in destructive testing experiments. In such settings, estimation based on record-like observations or extreme order statistics can lead to substantial savings in resources. 
For records, Glick \cite{Glick78} showed that the sequence of records can be obtained without fully measuring (i.e., destroying) every observation. Specifically, a unit is fully measured only if it exceeds the current record level, that is, only if it becomes a new record. Geometric records can be obtained through a slight modification of this procedure, in which the threshold is set to $\delta$ times the current record. For Hill's estimator, the threshold is taken to be the $(k+1)$-th largest order statistic. Accordingly, we include the \emph{effective sampling size}, defined as the total number of fully measured units under this sampling scheme.

For the simulation study, we fix $n=10$ records, $\gamma\in\{1,2,3\}$, and several values of the parameters $\delta$, $k$, and $\ell$. We then simulate $10{,}000$ sequences of i.i.d.~samples following the Pareto, Fr\'{e}chet, and log-logistic distributions for each value of the tail index and the tuning parameters. From each sequence, all estimators are computed, allowing us to provide an estimate of their expectations and MSEs. Results are presented in Table \ref{table_mse}.  

Several interesting conclusions can be drawn from this simulation study.

\begin{itemize}
	\item Our estimator exhibits a slight negative bias in all cases. It is noteworthy that, for any fixed distribution and value of $\gamma$, the bias remains fairly stable across the considered values of $\delta$. In contrast, the MSE decreases as $\delta$ decreases, as expected due to the larger amount of information provided by the geometric near-records. Note, however, that taking $\delta$ too small significantly increases the effective sampling size, which translates into higher costs in destructive testing settings; moreover, it may potentially lead to the inclusion of non-extreme observations in the estimator, thereby yielding unreliable estimates of the tail index.
	
	\item The estimator shows a reasonable robustness when applied to distributions other than the Pareto model. In particular, its performance under the Fr\'{e}chet and log-logistic distributions remains comparable to that observed in the Pareto case. This suggests that the proposed estimator is not too sensitive to deviations from the exact Pareto distribution.
	
	\item When compared to Hill's estimator, Table \ref{table_mse} shows that, for a comparable MSE, our estimator requires a substantially smaller effective sample size. For instance, in the case of the $\dpareto(2;1)$ distribution, comparable bias and MSE values are obtained for $\delta=0.6$ (mean $=1.902$, MSE $=0.139$), corresponding to a median number of 26 geometric records, and for $k=30$ (mean $=2.072$, MSE $=0.156$), which relies on a median of 330 observations. Hence, similar (and even improved) accuracy is achieved by our estimator using substantially fewer effective observations. This indicates that the information extracted from geometric records is used more efficiently than that contained in the largest order statistics. A similar pattern is observed across the remaining settings.
	
	\item Regarding Berred's estimators, it is noteworthy that our estimator obtains smaller MSE in all but three cases, in which our estimator and second Berred's estimator achieve very close values (see the cases $\gamma=1$, $\delta=0.8$, $\ell=3$ for the three distributions in Table \ref{table_mse}). Finally, note that our estimator has smaller absolute bias than both Berred's estimators. Overall, our estimator tends to improve upon the record-based alternatives. This is consistent with the fact that our estimator uses not only records (as Berred's estimators do) but also geometric near-records, allowing for a flexible incorporation of information through the choice of $\delta$. 
	
\end{itemize}
{
\captionsetup{font=footnotesize}
\begin{sidewaystable}[htbp]
	\footnotesize 
	\setlength{\tabcolsep}{4.5pt} 
	\centering 
	\caption{Estimated expectation and MSE of our estimator $\hat{\gamma}_{\delta,m,n}^A$ in \eqref{eq_estimator_gamma}, Hill's estimator $\hat{\gamma}_{k}^{\text{H}}$, and Berred's estimators $\hat{\gamma}_{\ell,n}^{\text{B1}}$ and $\hat{\gamma}_{\ell,n}^{\text{B2}}$. The estimates of the MSE were computed from 10{,}000 independent replications based on i.i.d.~samples with $R_1>A=5$, generated until $n=10$ records (together with their associated geometric near-records) were observed. All results were obtained for $m=5$. In parentheses, we report the median effective sampling size of our estimator and Hill's estimator.} 
	\begin{tabular}{cccccccccccc} \toprule & $\delta$ & $\espe(\hat{\gamma}^A_{\delta,m,n})$ & $\operatorname{MSE}(\hat{\gamma}^A_{\delta,m,n})$ & $k$ & $\espe(\hat{\gamma}^{\text{H}}_{k})$ & $\operatorname{MSE}(\hat{\gamma}^{\text{H}}_{k})$ & $\ell$ & $\espe (\hat{\gamma}^{\text{B1}}_{\ell,n})$ & $\operatorname{MSE}(\hat{\gamma}^{\text{B1}}_{\ell,n})$ & $\espe(\hat{\gamma}^{\text{B2}}_{\ell,n})$ & $\operatorname{MSE}(\hat{\gamma}^{\text{B2}}_{\ell,n})$\\ 
	\midrule 
	\multirow{5}{1.5cm}{$\dpareto(1; 1)$} & 0.8 & 0.984 (12) & 0.077 (12) & 5 & 1.261 (62) & 0.576 (62) & 3 & 1.530 & 2.077 & 0.917 & 0.075 \\ 
	& 0.6 & 0.964 (16) & 0.055 (16) & 10 & 1.113 (110) & 0.172 (110) & 4 & 1.334& 0.875 & 0.908 & 0.078 \\ 
	& 0.5 & 0.958 (19) & 0.047 (19) & 15 & 1.073 (156) & 0.094 (156) & 5 & 1.254& 0.547 & 0.900 & 0.078 \\ 
	& 0.4 & 0.956 (24) & 0.038 (24) & 20 & 1.052 (200) & 0.065 (200) & 6 & 1.204& 0.406 & 0.887& 0.078 \\ 
	& 0.2 & 0.956 (47) & 0.022 (47) & 30 & 1.037 (281) & 0.041 (281) & 7 & 1.172 & 0.278 & 0.878& 0.079 \\ 
	\midrule 
	\multirow{5}{1.5cm}{$\dpareto(2; 1)$} & 0.8 & 1.941 (15) &0.232 (15) & 5 & 2.496 (69) &2.191 (69) & 3 & 2.975 &10.616 & 1.562 & 0.337 \\ 
	& 0.6 & 1.902 (26) &0.139 (26) & 10 & 2.218 (124) &0.651 (124) & 4 & 2.709 &4.646 & 1.538 & 0.354 \\ & 0.5 & 1.906 (38) &0.105 (38) & 15 & 2.141 (177) &0.372 (177) & 5 & 2.486 & 2.270 & 1.509 & 0.366 \\ 
	& 0.4 & 1.912 (58) &0.074 (58)& 20 & 2.109 (231) &0.258 (231) & 6 & 2.389 & 1.576 & 1.485 & 0.397 \\
	 & 0.2 & 1.953 (229) &0.023 (229) & 30 & 2.072 (330) & 0.156 (330) & 7 & 2.344 & 1.252 & 1.439 & 0.420 \\ 
	 \midrule \multirow{5}{1.5cm}{$\dpareto(3; 1)$} & 0.8 & 2.880 (19) &0.419 (19) & 5 & 3.758 (75)&5.183 (75) & 3 & 4.549 & 21.715 & 2.041 & 1.117 \\ 
	 & 0.6 & 2.865 (44) &0.210 (44) & 10 & 3.362 (138)&1.602 (138) & 4 & 4.012 & 9.279 & 1.990 & 1.153 \\ 
	 & 0.5 & 2.877 (74) &0.137 (74) & 15 & 3.204 (200) &0.834 (200) & 5 & 3.767 & 5.065 & 1.950 & 1.248 \\ 
	 & 0.4 & 2.912 (145) &0.077 (145)& 20 & 3.160 (260) &0.579 (260) & 6 & 3.576 & 3.454 & 1.895 & 1.348 \\ 
	 & 0.2 & 2.976 (1153) & 0.010 (1153) & 30 & 3.102 (376) &0.352 (376) & 7 & 3.501 & 2.723 & 1.847 & 1.463 \\ 
	 \midrule \addlinespace \midrule \multirow{5}{1.5cm}{$\dfrechet(1)$} & 0.8 & 0.970 (12) &0.075 (12) & 5 & 1.252 (62) & 0.593 (62) & 3 & 1.490 &2.094 & 0.904 & 0.073\\ 
	 & 0.6 & 0.958 (16) &0.054 (16) & 10 & 1.104 (111) & 0.159 (111) & 4 & 1.329 &0.963 & 0.899 & 0.075\\ 
	 & 0.5 & 0.953 (19) &0.045 (19) & 15 & 1.072 (157) & 0.094 (157) & 5 & 1.248 &0.514 & 0.890 & 0.075\\ 
	 & 0.4 & 0.946 (24) &0.037 (24)& 20 & 1.050 (201) & 0.064 (201) & 6 & 1.216 &0.375 & 0.878 & 0.076\\ & 0.2 & 0.940 (46) &0.023 (46) & 30 & 1.033 (287)& 0.038 (287) & 7 & 1.160 &0.301 & 0.867 & 0.078\\ \midrule \multirow{5}{1.5cm}{$\dfrechet(2)$} & 0.8 & 1.943 (15) &0.229 (15) & 5 & 2.483 (68)& 2.136 (68) & 3 & 3.002 &10.679 & 1.551 & 0.339 \\
	  & 0.6 & 1.901 (26) &0.139 (26) & 10 & 2.232 (124) & 0.667 (124) & 4 & 2.705 &3.984 & 1.534 & 0.355 \\ 
	  & 0.5 & 1.897 (37) &0.104 (37) & 15 & 2.138 (178) & 0.367 (178) & 5 & 2.497 &2.337 & 1.508 & 0.374 \\ 
	  & 0.4 & 1.900 (58) &0.077 (58)& 20 & 2.100 (230) & 0.258 (230) & 6 & 2.405 &2.405 & 1.477& 0.393 \\ 
	  & 0.2 & 1.924 (223) &0.026 (223) & 30 & 2.069 (331)& 0.160 (331) & 7& 2.325 & 1.188 & 1.444 & 0.421 \\
	   \midrule \multirow{5}{1.5cm}{$\dfrechet(3)$} & 0.8 & 2.878 (19) & 0.415 (19) & 5 & 3.770 (75) &5.241 (75) & 3 & 4.563 & 24.069 & 2.034 & 1.095 \\ 
	   & 0.6 & 2.869 (44) & 0.207 (44) & 10 & 3.335 (138)&1.454 (138) & 4 & 3.994 &8.650 & 2.003 & 1.185 \\ 
	   & 0.5 & 2.874 (74) & 0.138 (74) & 15 & 3.208 (201) &0.844 (201) & 5 & 3.765 & 5.168 & 1.948 & 1.274 \\ 
	   & 0.4 & 2.902 (143) & 0.077 (143)& 20 & 3.158 (261)&0.565 (261) & 6 & 3.616 &3.503 & 1.902 & 1.356 \\ 
	   & 0.2 & 2.932 (1110) &0.016 (1110) & 30 & 3.108 (379)&0.367 (379) & 7& 3.499 & 2.738& 1.852 & 1.467 \\ 
	   \midrule \addlinespace \midrule \multirow{5}{1.5cm}{$\dloglogistic(1)$} & 0.8 & 0.963 (12) & 0.072 (12) & 5 & 1.268 (63) & 0.624 (63) & 3& 1.478 & 2.019 & 0.899 & 0.072\\ 
	   & 0.6 & 0.946 (16) & 0.053 (16) & 10 & 1.114 (111) & 0.167 (111) & 4& 1.317 & 0.886 & 0.885 & 0.075\\ 
	   & 0.5 & 0.945 (19) & 0.044 (19) & 15 & 1.070 (158) & 0.092 (158) & 5& 1.246 & 0.557 & 0.881 & 0.077\\ 
	   & 0.4 & 0.938 (23) & 0.037 (23)& 20 & 1.050 (202) & 0.065 (202) & 6& 1.200 & 0.396 & 0.870 & 0.076\\ 
	   & 0.2 & 0.932 (45) & 0.024 (45) & 30 & 1.037 (289) & 0.040 (289) & 7& 1.170 & 0.303 & 0.856 & 0.079\\ 
	   \midrule \multirow{5}{1.5cm}{$\dloglogistic(2)$} & 0.8 & 1.932 (15) & 0.237 (15) & 5 & 2.490 (68.5) & 2.332 (68.5) & 3 & 3.032 & 10.208 & 1.554 &0.339 \\
	    & 0.6 & 1.898 (26) & 0.138 (26) & 10 & 2.231 (125) & 0.684 (125) & 4 & 2.682 & 4.343 & 1.532 &0.352 \\ 
	    & 0.5 & 1.896 (37) & 0.106 (37) & 15 & 2.137 (178) & 0.366 (178) & 5 & 2.510 & 2.481 & 1.498 &0.373 \\ 
	    & 0.4 & 1.895 (57) & 0.077 (57)& 20 & 2.105 (231) & 0.263 (231) & 6 & 2.403 & 1.587 & 1.474 &0.400 \\ 
	    & 0.2 & 1.904 (218) & 0.030 (218) & 30 & 2.069 (332) & 0.160 (332) & 7& 2.320 &1.160& 1.440 &0.426 \\ 
	    \midrule \multirow{5}{1.5cm}{$\dloglogistic(3)$} & 0.8 & 2.877 (19)&0.432 (19) & 5 & 3.744 (75)&4.989 (75) & 3 & 4.579 & 25.255 & 2.042 & 1.101 \\ 
	    & 0.6 & 2.843 (43)&0.211 (43) & 10 & 3.342 (139)&1.468 (139) & 4 & 3.977 & 8.642 & 2.003 & 1.178 \\ 
	    & 0.5 & 2.864 (74)&0.141 (74) & 15 & 3.201 (201)&0.817 (201) & 5 & 3.788 & 5.852 & 1.945 & 1.256 \\ 
	    & 0.4 & 2.900 (143)&0.077 (143)& 20 & 3.159 (260)&0.571 (260) & 6 & 3.599 & 3.376 & 1.904 & 1.363 \\ 
	    & 0.2 & 2.904 (1084.5)& 0.023 (1084.5) & 30 & 3.101 (378)&0.358 (378) & 7 & 3.498 & 2.786 &1.844 &1.461 \\
	     \bottomrule 
	     \end{tabular} 
	     \label{table_mse}
\end{sidewaystable}
}

In addition to the numerical results in Table \ref{table_mse}, another interesting characteristic of the proposed estimator is that it exhibits a smoother evolution of the estimates as data arrive sequentially than classical estimators such as Hill's estimator.

To illustrate this feature, we generate an i.i.d.~sample of size $N=10{,}000$ from the $\dloglogistic(3)$ distribution and compare Hill's estimates with those of our estimator in \eqref{eq_estimator_practical} as the sample size increases. Note that $N$ denotes the total number of observations rather than the number of records. We computed the trajectories of our estimator for $m=5$, $\delta\in\{0.5,0.4\}$, and $A=3$, and the corresponding trajectories of Hill's estimator. To ensure a meaningful comparison, we start counting observations once the first record exceeding $A$ is observed, and we choose $k$ so that, for each fixed $\delta$, the effective sampling sizes of both estimators are approximately the same, namely $k=2$ for $\delta=0.5$ and $k=5$ for $\delta=0.4$. Figure~\ref{figure_effective_sample_size_2} shows that our estimator converges quickly to the true value of $\gamma$ and exhibits smoother trajectories than Hill's estimator, without the abrupt jumps observed in the latter as new observations arrive.

\begin{figure}[htbp]
	\centering
	\caption{Sample paths of our estimator $\hat{\gamma}_{\delta,m,n}^A$ in \eqref{eq_estimator_practical} and Hill's estimator $\hat{\gamma}_{k}^{\text{H}}$ for a $\dloglogistic(3)$ sample, shown as functions of the sample size, with $m=5$, $\delta \in \{0.5,0.4\}$, $A = 3$, and $k \in \{2,5\}$.}
	\begin{subfigure}[t]{\textwidth}
		\centering
		\includegraphics[scale = 0.48]{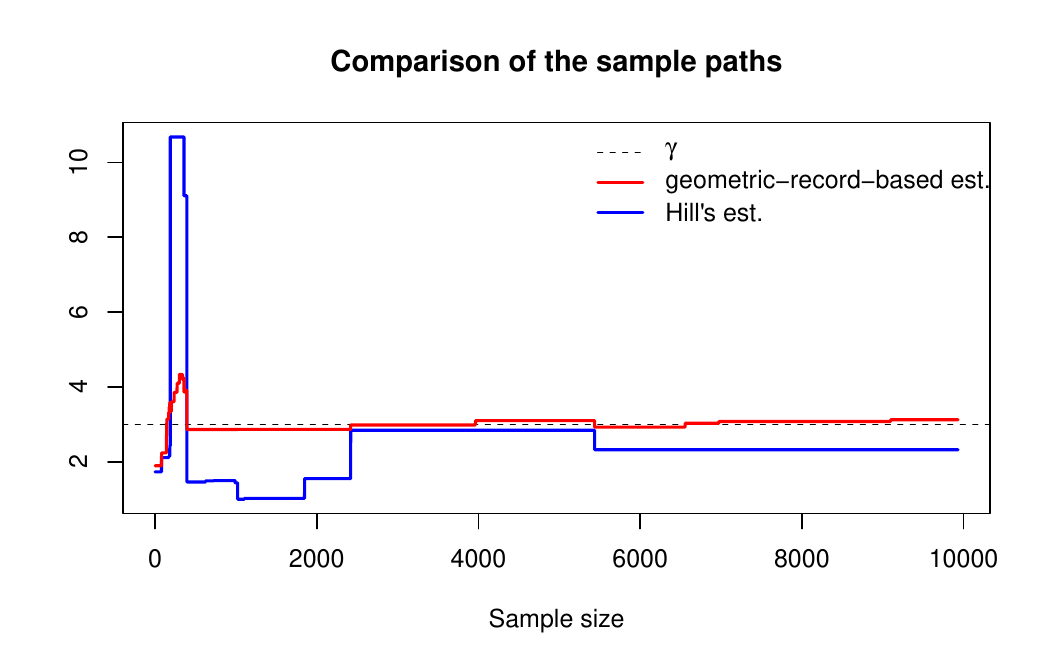}
		\caption{$\delta=0.5$ and $k=2$. The number of geometric records was $26$ and Hill's estimator used $24$ observations.}
	\end{subfigure}
	\begin{subfigure}[t]{\textwidth}
		\centering
		\includegraphics[scale = 0.48]{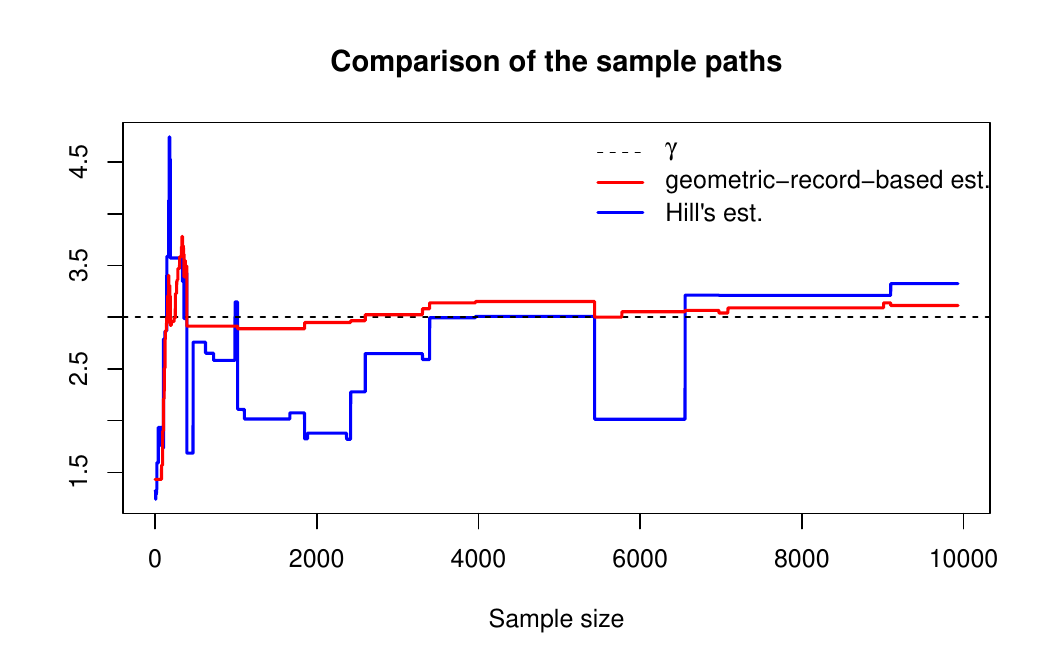}
		\caption{$\delta=0.4$ and $k=5$. The number of geometric records was $50$ and Hill's estimator used $58$ observations.}
	\end{subfigure}
	\label{figure_effective_sample_size_2}
\end{figure}

\subsection{Application to the study of the Dow Jones Industrial Average index}\label{section_data}

Since the seminal and highly influential works of \cite{Gopikrishnan98, Gopikrishnan99}, several studies have analyzed and confirmed the heavy-tailedness of the fluctuations in financial market returns, including recent papers such as \cite{Ahn24}. Empirical evidence allowed to establish a well-known universal law in economics: the so-called inverse cubic law. According to this law, the upper tail of the absolute standardized log-returns of a financial market index is approximately Pareto with tail index $\gamma=3$. 

Formally, given a finite time series $\{X_t:t=1,\ldots,T\}$ of the closing prices of a financial market index, its log-returns correspond to the time series
$$
Y_t := \log{(X_t)}-\log{(X_{t-1})},\quad 2\le t\le T,
$$
whence we obtain the standardized log-returns as:
$$
Z_t := \frac{Y_t - \overline{Y}}{S_Y},\quad 2\le t\le T,
$$
with $\overline{Y}$ and $S_Y$ being the sample mean and standard deviation of the entire series $\{Y_t:t=1,\ldots,T\}$, respectively. In this setting, the inverse cubic law \cite{Gopikrishnan98} establishes
$$
\prob(|Z_t| > x) \sim x^{-3}\;\text{ as } x\to\infty.
$$

We analyze the heaviness of the tail of the daily standardized log-returns of the Dow Jones Industrial Average index (DJI) from January 1, 1985, to December 31, 2025. We extracted the data from the website Investing.com; see \href{https://www.investing.com/indices/us-30}{https://www.investing.com/indices/us-30}. We apply our estimator $\hat{\gamma}^A_{\delta,m,n}$ in  \eqref{eq_estimator_practical} and \eqref{eq_estimator_beta_practical} in order to assess whether our results align with the existing literature in finance. 

For this purpose, we first note that, if the DJI satisfies the inverse cubic law, then
$$
\log{\prob(|Z_t| > x)} \approx -3\log{(x)}\;\text{ for large }x.
$$
Accordingly, we can plot the logarithm of the empirical survival function (ESF) of $|Z_t|$ against $\log{(x)}$ in order to detect from which threshold onward the linear dependence holds. As shown in Figure \ref{plot_esf_dj}, the linear behavior starts at $x \approx 1.5$. A regression fit for $x\ge 1.5$ yields $-0.310 -3.387 \log{(x)}$.
\begin{figure}[htbp]
	\centering
	\caption{Logarithm of the empirical survival function of the absolute standardized log-returns of the Dow Jones index.}
	\includegraphics[scale = 0.5]{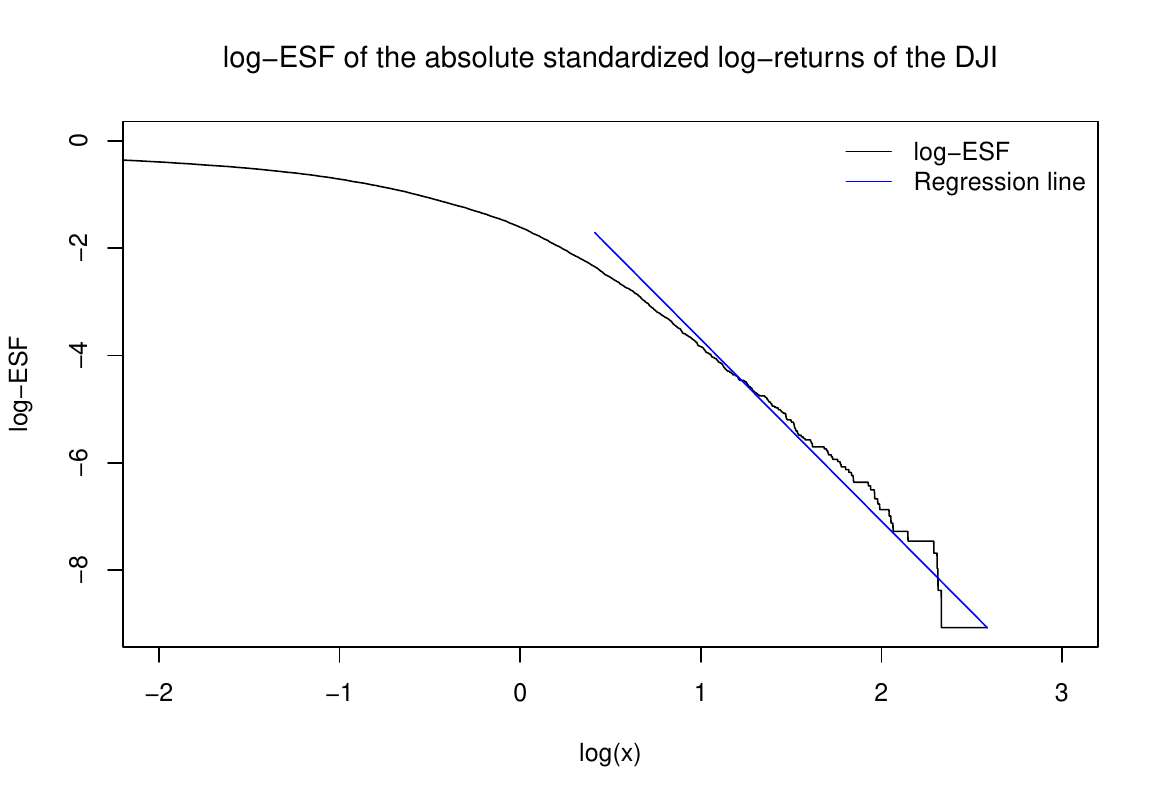}
	\label{plot_esf_dj}
\end{figure}

This analysis suggests a suitable choice for the censoring parameter $A$, namely $A=1.5$. We also fix $m=5$, and then compute our estimator for several values of $\delta$ and their $95\%$ asymptotic confidence intervals (see  \eqref{confidence_interval_gamma}). The results are displayed in Figure \ref{plot_dj}, and some selected estimates are reported in Table \ref{table_dj}.

\begin{figure}[h!]
	\centering
	\caption{Estimates of the tail index of the upper tail of the absolute standardized log-returns of the Dow Jones index, together with their $95\%$ asymptotic confidence intervals, for $\delta \in \{0.20, 0.21, 0.22, \ldots, 0.80\}$, $m = 5$, and $A=1.5$. The number of records in the sample exceeding $A$ is $n=9$.}
	\includegraphics[scale = 0.5]{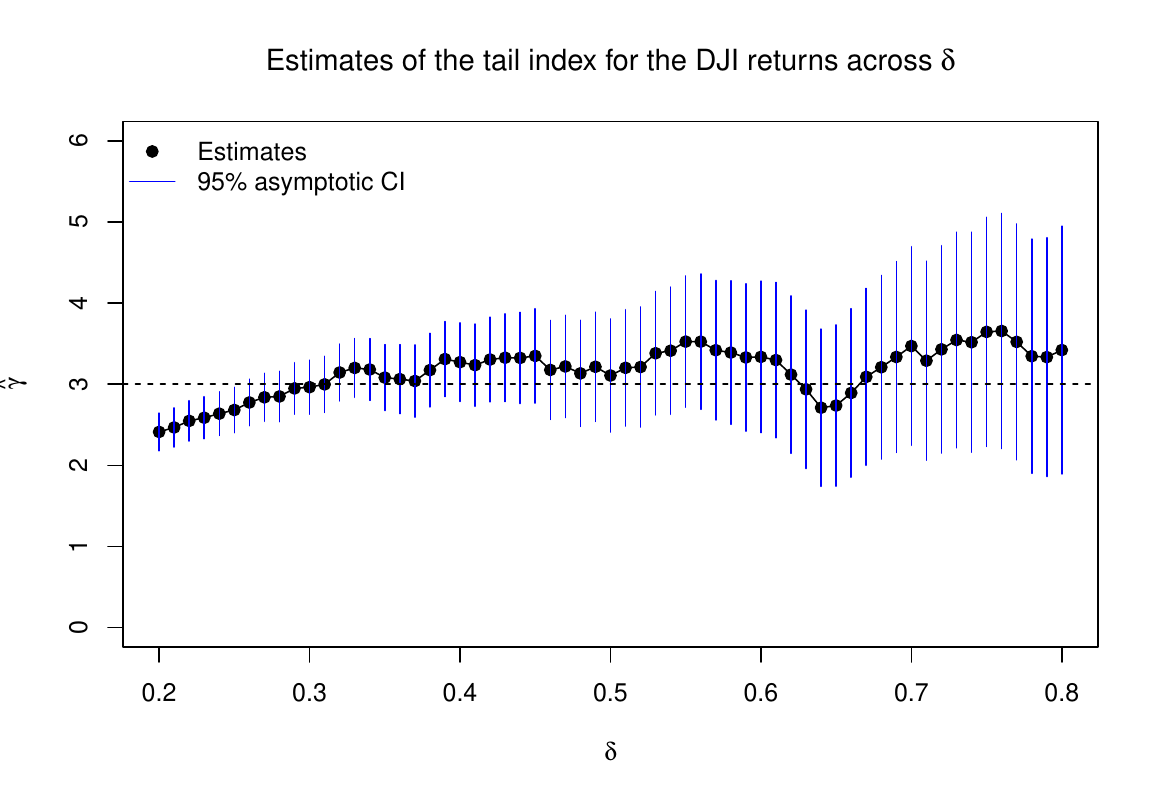}
	\label{plot_dj}
\end{figure}

\begin{table}[htbp]
	\centering
	\caption{Estimates of the tail index and $95\%$ asymptotic confidence intervals in Figure \ref{plot_dj}, for $\delta \in \{0.4, 0.5, 0.6\}$.}
	\begin{tabular}{ccc}
		\toprule
		$\delta$ & $\hat{\gamma}^A_{\delta,m,n}$ & $95\%$ asymptotic CI \\
		\midrule
		0.6 & 3.337 & $[2.403, 4.271]$\\
		0.5 & 3.107 & $[2.411, 3.804]$ \\
		0.4 & 3.272 & $[2.787, 3.757]$\\
		\bottomrule
	\end{tabular}
	\label{table_dj}
\end{table}

The evidence provided by Figure \ref{plot_dj} and Table \ref{table_dj} is consistent with the inverse cubic law. For all values of $\delta$ considered, our estimator yields estimates close to $3$, thereby supporting the empirical validity of this economic law. Moreover, it is noteworthy that for $0.26\le\delta\le 0.80$ the corresponding asymptotic confidence intervals contain the value $3$. As discussed in Section \ref{section_simulations}, choosing $\delta$ too small may result in the inclusion of non-extreme observations, which are not relevant for tail inference, and may yield unreliable estimates and confidence intervals.

\section{Conclusions and future work}\label{conclusion}
In this paper, we have introduced an estimator of the tail index for the Pareto and Pareto-type distributions based on geometric records.   We have proved its consistency and asymptotic normality for the general case of Pareto-type distributions under mild conditions on the tail.
	
Our simulation study indicates that the proposed estimator improves upon Berred’s estimators in terms of MSE. The comparison with Hill’s estimator depends on the tuning parameters $\delta$ in our method and $k$ in Hill’s estimator. When these are calibrated so that the effective number of fully measured units is comparable, our estimator exhibits improved estimation accuracy. In the context of destructive testing, where each full measurement entails a cost, this suggests cost savings relative to Hill’s estimator. 	
An additional advantage of our estimator over Hill’s lies in its behavior when data arrive sequentially, as is common in industrial and monitoring settings. Hill's estimator can exhibit rough trajectories as new data arrive, causing abrupt changes in the estimated tail index. By contrast, our estimator tends to produce smoother trajectories, making it particularly well-suited for sequential or streaming data where practitioners wish to update estimates continuously without large fluctuations, as in on-line monitoring settings.
	
Also, Monte Carlo simulations suggest that our estimator is robust to deviations from the Pareto distribution. Indeed, Table \ref{table_mse} indicates that the MSE for the other heavy-tailed distributions analyzed is comparable to that obtained under the Pareto case, supporting the applicability of the estimator in a wide range of practical settings.

We have applied our estimator to a real data set in an economic context, where heavy-tailed distributions are commonplace. The results are consistent with those reported in the literature and exhibit reasonable stability across different values of $\delta$. This illustrates the possible practical applicability of the proposed methodology for tail-index analysis.

A practical aspect of our estimator is the choice of $\delta$. The results in Table \ref{table_mse} show that, in general, smaller values of $\delta$ lead to more accurate estimates, as more information is incorporated. However, two points should be noted. First, choosing a small $\delta$ requires more data, which can be costly in destructive settings. Second, smaller values of $\delta$ lead to the inclusion of observations that are far from the extremes, i.e., from regions of the distribution that are not truly in the tail.  Based on extensive Monte Carlo simulations across different heavy-tailed distributions, values around $\delta \approx 0.4$ provide a reasonable balance between efficiency and data usage.

Overall, the results suggest that geometric-record-based methods provide a flexible alternative for tail index estimation in settings where data arrive sequentially or full observation is costly, complementing existing extreme-value approaches.

We conclude by outlining some directions for future work:

\begin{itemize}
	\item In the present paper, all geometric records contribute equally to the estimator. A natural refinement would be to assign higher weights to geometric near-records associated with the most recent records, which are more likely to lie deeper in the tail. While this modification may be suboptimal for the exact Pareto case, it may improve performance for other Pareto-type distributions, since observations corresponding to the first records may not yet belong to the tail, and should therefore receive less weight in the estimation. A thorough study is required both to provide guidelines for the choice of the weights and to establish the asymptotic properties of the resulting estimators.

	\item Another potential improvement of our estimator would be to use an adaptive choice of $\delta$. Allowing $\delta$ to vary as records are observed could help regulate the number of geometric records collected over time.
\end{itemize}

\section*{Acknowledgments}

This research was partially funded by grant PID2023-150234NB-I00 from the Ministry of Science, Innovation and Universities of Spain. The authors are members of the research group Modelos Estoc\'{a}sticos (E46\_23R), Gobierno de Arag\'{o}n. M. Alcalde gratefully acknowledges the support by the doctoral scholarship ORDEN ECU/592/2024 from Gobierno de Arag\'{o}n.

\begin{appendices}
\section{}
\begin{lemma}\label{lemma_truncated_geometric}
	Let $p\in(0,1)$ and $n\in\N$. If $X$ follows a truncated geometric distribution to  $\{0,1,\ldots,n\}$, then
	\begin{enumerate}[a)]
		\item $$
		\espe(X) = \frac{1-p}{p}-(n+1)\frac{(1-p)^{n+1}}{1-(1-p)^{n+1}}.
		$$ \label{lemma_truncated_geometric_expectation}
		\item $$
		\var(X) = \frac{1-p}{p^2}-(n+1)^2\frac{(1-p)^{n+1}}{(1-(1-p)^{n+1})^2}.
		$$
		\label{lemma_truncated_geometric_variance}
	\end{enumerate}
	\begin{proof}
		This is an elementary exercise that can be found in many textbooks. A direct (but somewhat tedious) proof consists in computing the expectation and variance from their definitions using the probability mass function of the truncated geometric distribution, $
		\prob(X=k)=\frac{p(1-p)^k}{1-(1-p)^{n+1}}$, $k=0,\ldots,n$.
		A more elegant argument uses conditional calculus. Let $Z\sim \dgeom(p)$ and define $Y=\mathbf{1}(Z>n)$. Then $
		\espe(X)=\espe(Z\mid Y=0)$, $\var(X)=\var(Z\mid Y=0)$ and note that, conditional on $\{Y=1\}$, $Z$ is equal in distribution to $n+1+Z$. The result then follows by expressing the expectation and variance of $Z$ conditioning on $Y$.
	\end{proof}
\end{lemma}

\begin{lemma} \label{prop_exponential_appendix}
	Let $\{E_n\}_{n\in\N}$ be a sequence of i.i.d.~random variables following an exponential distribution with rate $1$. Then, $	\limsup_{n\to\infty} E_n/\log{(n)} = 1$.
\end{lemma}
\begin{proof}
	Let $k \in\N$ be fixed. For each $n\in\N$, we define the events $A_n:= \{E_n>(1+1/k)\log{(n)}\}$ and $B_n:= \{E_n>(1-1/k)\log{(n)}\}$. By the first Borel-Cantelli Lemma, it is straightforward to see that $\prob(A_n\text{ i.o.})=0$ and, therefore, $\limsup_{n\to\infty} (\log{(n)})^{-1}E_n\le 1+1/k$ a.s. Also, the second Borel-Cantelli Lemma allows us to conclude that $\prob(B_n\text{ i.o.})=1$, which yields $\limsup_{n\to\infty} (\log{(n)})^{-1}E_n\ge 1-1/k$ a.s. Thus $\limsup_{n\to\infty} (\log{(n)})^{-1}E_n\in(1-1/k,1+1/k)$ a.s.~for all $k\in\N$, and the result follows.
\end{proof}

\begin{lemma}\label{lemma_liminf_exponential}
	Let $\{E_n\}_{n\in\N}$ be a sequence of i.i.d.~random variables following an exponential distribution with rate $\lambda >0$. Also, let $\{x\}$ denote the fractional part of $x$, for all $x\in\R$, and $\{a_n\}_{n\in\N}$ be a sequence of positive numbers such that $a_n\to\infty$ as $n\to\infty$. Then, almost surely, $\liminf_{n\to\infty} a_n\{E_n\} = c$,
	with $c = 0,\infty$. Moreover, $\liminf_{n\to\infty} a_n\{E_n\} = 0$ a.s.~if and only if $\sum_{n=1}^\infty1/a_n = \infty$.
\end{lemma}
\begin{proof}
	Let $\eta>0$ and $n_0\in\N$ such that $a_n^{-1}\eta<1$ for all $n\ge n_0$. Then,
	\begin{align*}
		\prob(a_n\{E_n\}\le \eta) &= \sum_{k=0}^\infty\prob(\{E_n\}\le a_n^{-1}\eta, E_n\in(k,k+1))\\ &=\sum_{k=0}^\infty\prob(E_n\in(k, k + a_n^{-1}\eta))\\
		&= \frac{1 - \e^{-\frac{\lambda \eta}{a_n} }}{1-\e^{-\lambda}}\\ &\sim \frac{\lambda\eta}{(1-\e^{-\lambda})a_n}\;\text{ as }\; n\to\infty.
	\end{align*}
	Hence, if $\sum_{n\ge 1} 1/a_n<\infty$, the first Borel-Cantelli Lemma implies $\liminf_{n\to\infty} a_n\{E_n\} > \eta$ a.s. When $\sum_{n\ge 1} 1/a_n=\infty$, the second Borel-Cantelli Lemma yields $\liminf_{n\to\infty} a_n\{E_n\} \le \eta$ a.s. Since the choice of $\eta$ is arbitrary, it follows that $\liminf_{n\to\infty} a_n\{E_n\} = \infty$ a.s.~if $\sum_{n\ge 1} 1/a_n<\infty$ and $\liminf_{n\to\infty} a_n\{E_n\}= 0$ a.s.~if $\sum_{n\ge 1} 1/a_n=\infty$.
\end{proof}
\end{appendices}

\begin{thebibliography}{99}
	
	\bibitem[Ahn \emph{et al.}(2024)]{Ahn24} Ahn, K., Cong, L., Jang, H. \emph{et al.} (2024). Business cycle and herding behavior in stock returns: theory and evidence. \emph{Financ. Innov.}, \textbf{10}, 6. DOI: \href{https://doi.org/10.1186/s40854-023-00540-z}{https://doi.org/10.1186/s40854-023-00540-z}.
	
	\bibitem[Arnold \emph{et al.}(1998)]{Arnold98} Arnold, B. C., Balakrishnan, N., \& Nagaraja, H. N. (1998). \emph{Records}. John Wiley \& Sons. DOI: \href{https://doi.org/10.1002/9781118150412}{https://doi.org/10.1002/9781118150412}.
	
		\bibitem[Azeem \emph{et al.}(2025)]{Azeem25} 
		Azeem, R., Aslam, M., \& Mehmood, T. (2025). Bayesian and classical methods for log logistic distribution: exploring upper record values in bladder cancer remission times. \emph{J. Stat. Theory Appl.}, \textbf{24}, 570--589. DOI: \href{https://doi.org/10.1007/s44199-025-00114-1}{https://doi.org/10.1007/s44199-025-00114-1}.
	
	
	\bibitem[Begu\v{s}i\'{c} \emph{et al.}(2018)]{Begusic18} Begu\v{s}i\'{c}, S., Kostanj\v{c}ar, Z., Stanley, H. E., \& Podobnik, B. (2018). Scaling properties of extreme price fluctuations in Bitcoin markets. \emph{Physica A: Statistical Mechanics and its Applications}, \textbf{510}, 400--406. DOI: \href{https://doi.org/10.1016/j.physa.2018.06.131}{https://doi.org/10.1016/j.physa.2018.06.131}.
	
	\bibitem[Berred(1992)]{Berred92} Berred, M. (1992). On record values and the exponent of a distribution with regularly varying upper tail. \emph{J. Appl. Prob.}, \textbf{29}(3), 575--586. DOI: \href{https://doi.org/10.2307/3214894}{https://doi.org/10.2307/3214894}.
	
	
	\bibitem[Bingham \emph{et al.}(1989)]{Bingham89} Bingham, N. H., Goldie, C. M., \& Teugels, J. L. (1989). \emph{Regular Variation}. Encyclopedia of Mathematics and its Applications. Cambridge University Press.
	
	\bibitem[Chaudhry \emph{et al.}(2025)]{Chaudhry25}Chaudhry, S. M., Chen, X. H., Ahmed, R., \& Nasir, M. A. (2025). Risk modelling of ESG (environmental, social, and governance), healthcare, and financial sectors. \emph{Risk Analysis}, \textbf{45}, 477--495. DOI: \href{https://doi.org/10.1111/risa.14195}{https://doi.org/10.1111/risa.14195}.
	
	\bibitem[Chen \& Liu(2017)]{Chen17} Chen, Q., \& Liu, J. (2017). The conditional Borel-Cantelli lemma and applications. \emph{J. Korean Math. Soc.}, \textbf{54}(2), 441--460. DOI: \href{https://doi.org/10.4134/JKMS.j160036}{https://doi.org/10.4134/JKMS.j160036}.
	
	
	\bibitem[Davletov(2022)]{Davletov22} Davletov, F. (2022). Estimating the tail index of conditional distribution of asset returns. \emph{International Journal of Financial Research}, \textbf{13}(2), 21615. DOI: \href{https://doi.org/10.5430/ijfr.v13n2p14}{https://doi.org/10.5430/ijfr.v13n2p14}. 
	
	
	\bibitem[Eliazar(2005)]{Eliazar2005} Eliazar, I.  (2005). On geometric record times. \emph{Physica A: Stat. Mech. Appl.}, \textbf{348}, 181-198. DOI: \href{https://doi.org/10.1016/j.physa.2004.09.009}{https://doi.org/10.1016/j.physa.2004.09.009}.
	
					\bibitem[Empacher \emph{et al.}(2025)]{Empacher25} Empacher, C., Kamps, U., \& Schmiedt, A. B. (2025). Prediction intervals for future Pareto record claims. \emph{Eur. Actuar. J.}, \textbf{15}, 163--197. DOI: \href{https://doi.org/10.1007/s13385-024-00397-1}{https://doi.org/10.1007/s13385-024-00397-1}.
					
	
	\bibitem[Fedotenkov(2020)]{Fedotenkov20} Fedotenkov, I. (2020). A review of more than one hundred Pareto-tail index estimators. \emph{Statistica}, \textbf{80}(3), 245--299. DOI: \href{https://doi.org/10.6092/issn.1973-2201/9533}{https://doi.org/10.6092/issn.1973-2201/9533}.
	
	\bibitem[Ferguson(1996)]{Ferguson96} Ferguson, T. S. (1996). \emph{A Course in Large Sample Theory}. Texts in Statistical Science. Chapman \& Hall/CRC.
	
	\bibitem[Glick(1978)]{Glick78} Glick, N. (1978). Breaking records and breaking boards. \emph{The American Mathematical Monthly}, \textbf{85}(1), 2--26. DOI: \href{https://doi.org/10.2307/2978044}{https://doi.org/10.2307/2978044}.

	\bibitem[Gomes \emph{et al.}(2020)]{Gomes20} Gomes, M. I., Caeiro, F., Figueiredo, F., Henriques-Rodrigues, L., \& Pestana, D. (2020). Corrected-Hill versus partially reduced-bias value-at-risk estimation. \emph{Communications in Statistics-Simulation and Computation}, \textbf{49}(4), 867--885. DOI: \href{https://doi.org/10.1080/03610918.2018.1489053}{https://doi.org/10.1080/03610918.2018.1489053}.
	
	\bibitem[Gopikrishnan \emph{et al.}(1998)]{Gopikrishnan98} Gopikrishnan, P.,  Meyer, M., Amaral, L. A. N., \& Stanley, H. E. (1998). Inverse cubic law for the distribution of stock price variations. \emph{Eur. Phys. J. B}, \textbf{3}, 139--140. DOI: \href{https://doi.org/10.1007/s100510050292}{https://doi.org/10.1007/s100510050292}.
	
	\bibitem[Gopikrishnan \emph{et al.}(1999)]{Gopikrishnan99} Gopikrishnan, P., Plerou, V., Amaral, L. A. N., Meyer, M., \& Stanley, H. E. (1999). Scaling of the distribution of fluctuations of financial market indices. \emph{Physical Review E}, \textbf{60}(5), 5305--5316. DOI: \href{https://doi.org/10.1103/PhysRevE.60.5305}{https://doi.org/10.1103/PhysRevE.60.5305}.
	
			
	\bibitem[Gouet \emph{et al.}(2014)]{Gouet14} Gouet, R., L\'opez, F. J., Maldonado, L., \& Sanz, G. (2014). Statistical inference for the geometric distribution based on 
	$\delta$-records. \emph{Computational Statistics \& Data Analysis}, \textbf{78}, 21--32. DOI: \href{https://doi.org/10.1016/j.csda.2014.04.002}{https://doi.org/10.1016/j.csda.2014.04.002}.
	
		\bibitem[Gouet \emph{et al.}(2012)]{Gouet12Geom} Gouet, R., L\'opez, F. J., \& Sanz, G. (2012). On geometric records: rate of appearance and magnitude. \emph{J. Stat. Mech.: Theory Exp.}, \textbf{2012}(01), P01005. DOI: \href{https://doi.org/10.1088/1742-5468/2012/01/P01005}{https://doi.org/10.1088/1742-5468/2012/01/P01005}.
	
	
	\bibitem[Hill(1975)]{Hill75} Hill, B. M. (1975). A simple approach to inference about the tail of a distribution. \emph{Ann. Stat.}, \textbf{3}(5), 1163--1174. DOI: \href{https://doi.org/10.1214/aos/1176343247}{https://doi.org/10.1214/aos/1176343247}.
	
		
	\bibitem[Langousis \emph{et al.}(2016)]{Langousis16} Langousis, A., Mamalakis, A., Puliga, M., \& Deidda, R. (2016). Threshold detection for the generalized Pareto distribution: Review of representative methods and application to the NOAA NCDC daily rainfall database. \emph{Water Resources Research}, \textbf{52}(4), 2659--2681. DOI: \href{https://doi.org/10.1002/2015WR018502}{https://doi.org/10.1002/2015WR018502}.
	
	\bibitem[Louzaoui \& El Arrouchi(2020)]{Louzaoui20} Louzaoui, A., \& El Arrouchi, M. (2020). On the maximum likelihood estimation of extreme value index based on $k$-record values. \emph{Journal of Probability and Statistics}, \textbf{2020}(2), 1--9. DOI: \href{https://doi.org/10.1155/2020/5497413}{https://doi.org/10.1155/2020/5497413}. 
	
	\bibitem[Louzaoui \& El Arrouchi(2023)]{Louzaoui23}Louzaoui, A., \& El Arrouchi, M. (2023). Improving the bias of a pseudo-maximum likelihood estimate of the extreme value index by $k$-records. \emph{J. Stat. Theory Appl.}, \textbf{22}, 54--69. DOI: \href{https://doi.org/10.1007/s44199-023-00055-7}{https://doi.org/10.1007/s44199-023-00055-7}.
	
	
	\bibitem[Nerantzaki \& Papalexiou(2022)]{Nerantzaki22} Nerantzaki, S. D., \& Papalexiou, S. M. (2022). Assessing extremes in hydroclimatology: A review on probabilistic methods. \emph{Journal of Hydrology}, \textbf{605}, 127302. DOI: \href{https://doi.org/10.1016/j.jhydrol.2021.127302}{https://doi.org/10.1016/j.jhydrol.2021.127302}.
	
	\bibitem[Nolan(2020)]{Nolan20} Nolan, J. P. (2020). \emph{Univariate Stable Distributions}. Springer Series in Operations Research and Financial Engineering. Springer Nature. DOI: \href{https://doi.org/10.1007/978-3-030-52915-4}{https://doi.org/10.1007/978-3-030-52915-4}.

	
	\bibitem[Pickands(1975)]{Pickands75} Pickands, J. III (1975). Statistical inference using extreme order statistics. \emph{Ann. Statist.}, \textbf{3}(1), 119--131. DOI: \href{https://doi.org/10.1214/aos/1176343003}{https://doi.org/10.1214/aos/1176343003}.
	
		
	\bibitem[Schmiedt \emph{et al.}(2025)]{Schmiedt25} Schmiedt, A. B., Empacher, C., \& Kamps, U. (2025). One- and two-sided prediction intervals for future Pareto record values with applications. \emph{J. Stat. Theory Appl.}, \textbf{24}, 489--514. DOI: \href{https://doi.org/10.1007/s44199-025-00119-w}{https://doi.org/10.1007/s44199-025-00119-w}.
	
	\bibitem[She \emph{et al.}(2025)]{She25} She, R., Dai, L., \& Ling, S. (2025). Testing for change-points in heavy-tailed time series---A Winsorized CUSUM approach. \emph{Journal of Business \& Economic Statistics}, 1--13. DOI: \href{https://doi.org/10.1080/07350015.2025.2561747}{https://doi.org/10.1080/07350015.2025.2561747}.
	
	\bibitem[Shorrock(1972)]{Shorrock72} Shorrock, R. W. (1972). On record values and record times. \emph{Journal of Applied Probability}, \textbf{9}(2), 316--326. DOI: \href{https://doi.org/10.2307/3212801}{https://doi.org/10.2307/3212801}.
	
	\bibitem[Vogel \emph{et al.}(2024)]{Vogel24}Vogel, R. M., Papalexiou, S. M., Lamontagne, J. R., \& Dolan, F. C. (2024). When heavy tails disrupt statistical inference. \emph{The American Statistician}, \textbf{79}(2), 221--235. DOI: \href{https://doi.org/10.1080/00031305.2024.2402898}{https://doi.org/10.1080/00031305.2024.2402898}.
	
\end{thebibliography}
\end{document}